\documentclass[11pt]{amsart}
\usepackage[T1]{fontenc}
\usepackage[utf8]{inputenc}
\usepackage[english]{babel}
\usepackage{lmodern}
\usepackage[margin=1in]{geometry}
\usepackage{amsmath,amssymb,mathtools,mathrsfs}
\usepackage{graphicx,booktabs,tabularx,enumitem}
\usepackage{microtype,placeins,needspace}
\usepackage[foot]{amsaddr}
\usepackage{cite}
\usepackage[hidelinks]{hyperref}
\usepackage[nameinlink,capitalize]{cleveref}

\newtheorem{theorem}{Theorem}[section]
\newtheorem{lemma}[theorem]{Lemma}

\newtheorem{proposition}[theorem]{Proposition}
\newtheorem{corollary}[theorem]{Corollary}
\theoremstyle{definition}

\newtheorem{remark}[theorem]{Remark}

\numberwithin{equation}{section}
\DeclareGraphicsExtensions{.pdf,.png,.jpg,.jpeg}
\graphicspath{{figures/}}
\setlist{nosep,leftmargin=*}
\allowdisplaybreaks
\newcommand{\proofstep}[2]{\par\smallskip\noindent\textbf{#1} (#2).\enspace\ignorespaces}
\newcommand{\R}{\mathbb{R}}
\newcommand{\Pcal}{\mathcal{P}}
\newcommand{\Ecal}{\mathcal{E}}

\newcommand{\Scal}{\mathcal{S}}
\newcommand{\Ical}{\mathcal{I}}

\DeclareMathOperator{\dist}{dist}
\DeclareMathOperator{\conv}{conv}
\DeclareMathOperator{\supp}{supp}
\DeclareMathOperator{\diam}{diam}
\DeclareMathOperator{\KL}{KL}
\DeclareMathOperator{\gph}{gph}
\DeclareMathOperator*{\argmin}{argmin}

\crefname{theorem}{Theorem}{Theorems}
\Crefname{theorem}{Theorem}{Theorems}
\crefname{lemma}{Lemma}{Lemmas}
\Crefname{lemma}{Lemma}{Lemmas}
\crefname{proposition}{Proposition}{Propositions}
\Crefname{proposition}{Proposition}{Propositions}
\crefname{corollary}{Corollary}{Corollaries}
\Crefname{corollary}{Corollary}{Corollaries}
\crefname{remark}{Remark}{Remarks}
\Crefname{remark}{Remark}{Remarks}
\crefname{example}{Example}{Examples}
\Crefname{example}{Example}{Examples}
\crefname{definition}{Definition}{Definitions}
\Crefname{definition}{Definition}{Definitions}

\title[Convergence of singular heat-gradient dynamics]{Convergence of Singular Dynamics Driven by Mixed Logarithmic Heat Gradients}
\author[K. Liu]{Kang Liu$^1$}
\address{$^1$Universit\'e Bourgogne Europe, CNRS, Institut de Math\'ematiques de Bourgogne, 21000 Dijon, France}
\email{kang.liu@u-bourgogne.fr}
\thanks{Corresponding author: Kang Liu.}
\author[E. Zuazua]{Enrique Zuazua$^{2,3,4}$}
\address{$^2$Friedrich--Alexander--Universit\"at Erlangen--N\"urnberg, Department of Mathematics, Chair for Dynamics, Control, Machine Learning, and Numerics, 91058 Erlangen, Germany}
\address{$^3$Universidad Aut\'onoma de Madrid, Departamento de Matem\'aticas, 28049 Madrid, Spain}
\address{$^4$Chair of Computational Mathematics, Fundaci\'on Deusto, 48007 Bilbao, Basque Country, Spain}
\email{enrique.zuazua@fau.de}
\thanks{
E. Zuazua was partially supported by the European Research Council (ERC) under the European Union's Horizon Europe research and innovation programme (grant agreement No.~101096251-CoDeFeL); by the Alexander von Humboldt Professorship program; the European Union's Horizon Europe MSCA project ModConFlex (HORIZON-MSCA-2021-DN-01, project 101073558); the Transregio 154 Project ``Mathematical Modelling, Simulation and Optimization Using the Example of Gas Networks'' of the DFG; the AFOSR 24IOE027 project; the SURE-AI Norwegian Centre for Sustainable, Risk-Averse, and Ethical AI grant 357482, Research Council of Norway; by the Grant PID2023-146872OB-I00-DyCMaMod of MICIU (Spain), by the COST Actions CA24122 -- Multiscale Stochastics, Patterns, and Analysis of Combinatorial Environments and CA24136 -- Interactions between Control Theory and Machine Learning and by the Basque Government through the consolidated research group ``CMRG-Computational Mathematics Research Group" under Grant IT1856-26.
}
\date{}
\hypersetup{pdftitle={Convergence of Singular Dynamics Driven by Mixed Logarithmic Heat Gradients},pdfauthor={Kang Liu and Enrique Zuazua}}

\begin{document}

\begin{abstract}
We study the terminal behavior of singular deterministic and stochastic
dynamics driven by convex or extrapolative combinations of
logarithmic heat gradients. These equations arise in exact-score guided
diffusion sampling, including classifier-free guidance.
The rescaled potentials converge to weighted sums of squared support
distances, but extrapolative mixing can produce limiting directions
that depend on probability masses and lie outside the Clarke
subdifferential of the limiting potential.
A componentwise chain rule gives dissipation for the limiting
differential inclusion. Posterior relative entropy yields finite
deterministic kinetic action in logarithmic time, while a conditional
free-energy estimate bounds expected energy increments uniformly
in the restart time.
For compactly supported laws, these estimates give energy convergence
and equilibrium localization, almost surely for each fixed stochastic
initialization; extrapolation requires a lower small-ball mass bound
on the reference law.
For finite atomic measures, we prove that the equilibrium set is finite
for arbitrary site configurations. Every deterministic trajectory
therefore converges to an equilibrium, and stochastic convergence holds
almost surely for each fixed initialization, including endpoints at
projection interfaces.
At empirical endpoints with a unique nearest atom in each active support,
deterministic trajectories admit a square-root heat-time expansion with
an exponentially small remainder. The exponent $1/2$ is sharp when the
leading coefficient is nonzero.
Square-integrable perturbations in logarithmic time preserve deterministic
localization and empirical endpoint convergence for both dynamics. For nested
supports, the equilibrium geometry further constrains terminal samples
in classifier-free guidance.
These results provide a geometric and dynamical framework for
understanding terminal convergence in guided diffusion models.
\end{abstract}

\maketitle
\section{Introduction}\label{sec:introduction}

We study the terminal behavior of singular nonautonomous dynamics
driven by mixed logarithmic heat gradients. The problem is motivated
by guided diffusion sampling in generative AI, where samples are
constructed by reversing a noise-adding evolution. This reversal uses
the logarithmic gradients of the evolving densities, called scores
\cite{sohl2015deep,ho2020denoising,song2021scorebased}.
Under Gaussian noising, the heat equation regularizes the data laws,
and the resulting scores drive sampling in decreasing heat time
\cite{liuzuazua2025}. Guidance combines these scores to favor a prescribed
condition or compose information from different data laws
\cite{dhariwal2021diffusion,ho2021classifierfree,liu2022composable}.

Endpoint convergence asks whether an individual sampling trajectory has
a limiting output as the noise level tends to zero. When the trajectory
is stopped at a positive heat time, a convergence rate quantifies the
remaining truncation error. This pathwise question complements the study of
the generated distribution. It is mathematically delicate because the
scores and their spatial derivatives may diverge as heat time vanishes.
Consequently, well-posedness at every positive time does not by itself
provide a continuous extension to the terminal time.

We consider fixed convex and extrapolative combinations of two exact
heat-flow scores. Their rescaled potentials converge to weighted sums of squared
support distances, but convergence of the potentials alone does not
imply convergence of trajectories. Under extrapolation, limiting
directions may depend on probability masses and lie outside the Clarke
subdifferential of the limiting potential. We address this
difficulty by combining the geometry of the component supports with
entropy estimates for the original dynamics. The resulting analysis
gives equilibrium localization for compactly supported laws under the
stated hypotheses, and point convergence for arbitrary finite atomic
configurations, including endpoints at projection interfaces.

\subsection{Heat-flow framework and geometric formulation}
\label{subsec:intro-framework}

In the heat-flow framework~\cite{liuzuazua2025}, let
$u_{0,1},u_{0,2}$ be compactly supported probability measures and let
$u_k=G_t*u_{0,k}$ solve
\begin{equation}\label{eq:intro-heat}
 \partial_tu_k=\Delta u_k,\qquad
 u_k(\cdot,0)=u_{0,k},\qquad k=1,2,
\end{equation}
where $G_t$ is the Gaussian heat kernel with covariance $2tI_d$.
The regularized densities are smooth and positive for $t>0$, even when
the measures are singular. Their logarithmic gradients define the mixed score
\begin{equation}\label{eq:intro-guided-score}
 s^{(\lambda)}(x,t)
 =\lambda\nabla\log u_1(x,t)+(1-\lambda)\nabla\log u_2(x,t),
 \qquad\lambda\geq0.
\end{equation}
Equivalently,
\[
 h_\lambda=u_1^\lambda u_2^{1-\lambda},
 \qquad s^{(\lambda)}=\nabla\log h_\lambda.
\]
The product $h_\lambda$ is a weighted geometric mean for
$0\leq\lambda\leq1$ and a multiplicative extrapolation for $\lambda>1$,
with a negative power of the reference density. Thus mixing preserves
the logarithmic-gradient structure. It does not, however, generally
preserve the heat equation: a direct calculation gives
\[
 (\partial_t-\Delta)h_\lambda
 =\lambda(1-\lambda)|\nabla\log u_1-\nabla\log u_2|^2h_\lambda.
\]

We study the deterministic equation
\begin{equation}\label{eq:physical-flow}
 \frac{dX_t}{dt}=-s^{(\lambda)}(X_t,t),\qquad X_T=x_T,
\end{equation}
integrated from a fixed $T>0$ toward $t=0$. We also consider the
stochastic counterpart
\begin{equation}\label{eq:intro-stochastic-flow}
 dX_t^\varepsilon
 =-(1+\varepsilon)s^{(\lambda)}(X_t^\varepsilon,t)\,dt
       +\sqrt{2\varepsilon}\,dW_t,
 \qquad X_T^\varepsilon=x_T,\quad\varepsilon\geq0,
\end{equation}
with the same direction of time. Here $W$ is Brownian motion in decreasing
heat time, as specified in Section~\ref{subsec:stochastic-generation};
$\varepsilon=0$ recovers \eqref{eq:physical-flow}.
For $\lambda>1$, this is classifier-free guidance (CFG)
\cite{ho2021classifierfree} when the first law is conditional and the
second is the corresponding unconditional law; our analysis also permits unrelated
data laws, for which we use the term \emph{extrapolative score mixing}.

In the single-score setting, relative-entropy contraction follows
from the common Markov evolution of the reference and generated laws
\cite[Lemma~5.3]{liuzuazua2025}. This argument does not extend directly
to score mixtures, since $h_\lambda$ need not satisfy the heat equation.
We instead control the time variation of the mixed potential through
posterior relative entropy of the component heat flows, and stochastic
energy increments through a conditional free-energy estimate for the
transition densities.

Write $A_k=\supp u_{0,k}$ and introduce
\begin{equation}\label{eq:intro-Phi}
 \begin{aligned}
 F_\lambda(x,t)
 &=-4t\left(\lambda\log u_1(x,t)+(1-\lambda)\log u_2(x,t)\right),\\
 \Phi_\lambda(x)
 &=\lambda\dist(x,A_1)^2+(1-\lambda)\dist(x,A_2)^2.
 \end{aligned}
\end{equation}
A Laplace--Varadhan-type argument
\cite[Section~4.3]{DemboZeitouni1998} gives
$F_\lambda(\cdot,t)\to\Phi_\lambda$ locally uniformly as $t\to0^+$;
see \cref{lem:compact-soft-min}. Derivatives can nevertheless remain
singular.
Indeed, the covariance formula in \cref{lem:heat-hessian} exhibits a
$t^{-2}$ term in $D_x^2\log u_k$; after rescaling,
$D_x^2F_\lambda$ may still grow like $t^{-1}$ near projection ties.
These ties occur where a support has more than one nearest point.

To analyze the singular endpoint regime, introduce the similarity time
\[
 \tau=\log(T/t),\qquad t_\tau=Te^{-\tau},\qquad
 Y_\tau=X_{t_\tau},\qquad
 Y_\tau^\varepsilon=X_{t_\tau}^\varepsilon,
\]
with $Y_0=Y_0^\varepsilon=y_0=x_T$.
The time change preserves geometric paths and transforms $t\to0^+$
into $\tau\to\infty$. We state the trajectory results in similarity
time and then give their original-time interpretations.

\subsection{Main results}\label{subsec:intro-results}

Let $\Pi_k(x):=\argmin_{a\in A_k}|x-a|$ be the nearest-point set
of $A_k$. Define the set of nearest-support equilibria by
\[
 \Ecal_\lambda
 =\bigl\{x\in\R^d:
 x\in\lambda\conv\Pi_1(x)+(1-\lambda)\conv\Pi_2(x)\bigr\}.
\]
For $0\leq\lambda\leq1$, this set coincides with the Clarke
critical set of the limiting potential:
\[
 \Ecal_\lambda
 =\{x\in\R^d:0\in\partial^C\Phi_\lambda(x)\}.
\]
Here $\partial^C$ denotes the Clarke subdifferential, defined in
\cref{eq:Clarke-definition}. A support is called \emph{active}
when its score coefficient is nonzero; terms with zero coefficients
are omitted throughout.

The main results are as follows.
\begin{enumerate}[label=\textup{(\roman*)}]
\item \emph{Limiting support geometry} (\cref{thm:support-reduction}).
For compact active supports and every $\lambda\geq0$,
$F_\lambda(\cdot,t)\to\Phi_\lambda$ locally uniformly as $t\to0^+$.
For $0\leq\lambda\leq1$, we obtain the complete graphical convergence
\[
 \gph\nabla F_\lambda(\cdot,t)
 \longrightarrow\gph\partial^C\Phi_\lambda
 \qquad(t\to0^+).
\]
For $\lambda>1$, limiting directions lie in the separated envelope
$\lambda\partial^C d_1^2+(1-\lambda)\partial^C d_2^2$, where
$d_k=\dist(\cdot,A_k)$, but may lie outside
$\partial^C\Phi_\lambda$ and depend on the probability masses.

\item \emph{Deterministic localization for compact data}
(\cref{thm:compact-omega}).
For every initial state $y_0$, the logarithmic-time trajectory $Y$
has an energy limit $\ell\in\R$ and satisfies
\[
 \Phi_\lambda(Y_\tau)\longrightarrow\ell,
 \qquad
 \varnothing\neq\omega(Y)
 \subseteq\{x\in\Ecal_\lambda:\Phi_\lambda(x)=\ell\}.
\]
Here $\omega(Y)$ is the set of accumulation points as $\tau\to\infty$.
Thus localization places all accumulation points on one equilibrium
energy level, without asserting a single limit. Compactness suffices
for convex mixing; for extrapolation, we also assume a lower small-ball
mass bound on the negatively weighted reference law.

\item \emph{Deterministic endpoints and asymptotics for empirical data}
(\cref{thm:empirical-main}).
For finite atomic laws and every $\lambda\geq0$,
\[
 \#\Ecal_\lambda<\infty,
 \qquad Y_\tau\longrightarrow y^*\in\Ecal_\lambda
 \quad(\tau\to\infty)
\]
for every initialization and arbitrary site configurations, including
projection-interface endpoints. If each active support has a unique
nearest atom at $y^*$, we call it a \emph{strict-cell endpoint}. In this
case, for some $\xi^*\in\R^d$ and $\eta>0$,
\[
 Y_\tau=y^*+e^{-\tau/2}\xi^*
       +O\!\left(t_\tau e^{-\eta/t_\tau}\right)
 \qquad(\tau\to\infty).
\]
Consequently, $|Y_\tau-y^*|\leq Ce^{-\tau/2}$ for all $\tau\geq0$.
This is the square-root heat-time rate, sharp when $\xi^*\neq0$.

\item \emph{Stochastic localization and endpoints}
(\cref{thm:stochastic-convergence}).
Under the hypotheses of part~\textup{(ii)}, for each fixed initial
state $y_0$ and noise parameter $\varepsilon>0$, almost surely,
\[
 \Phi_\lambda(Y_\tau^\varepsilon)\longrightarrow\ell^\varepsilon,
 \qquad
 \dist(Y_\tau^\varepsilon,\Ecal_\lambda)\longrightarrow0.
\]
All accumulation points lie on the same equilibrium energy level.
For finite atomic laws, the whole trajectory converges almost surely
to a random equilibrium $y^{\varepsilon,*}\in\Ecal_\lambda$, including
the possibility of a projection-interface endpoint.
\end{enumerate}

The refined expansion and rate concern deterministic strict-cell
endpoints. Figure~\ref{fig:empirical-2d-convergence} also suggests
square-root heat-time decay of the stochastic root-mean-square distance
to the set of local minimizers in the empirical example.

Square-integrable perturbations in logarithmic time preserve deterministic
energy convergence and equilibrium localization, and empirical endpoint
convergence for both dynamics
(\cref{prop:L2-robustness,rem:stochastic-perturbations}).
In heat time, the condition becomes weighted square-integrability of the
score error along the perturbed sampling trajectory
(\cref{rem:physical-score-robustness}).

For classifier-free guidance with a positive-weight conditional component,
the supports are nested: $A_1\subseteq A_2$. When $\lambda>1$,
\cref{rem:nested} shows that the local and global minimizers of
$\Phi_\lambda$ are exactly $A_1$, while every equilibrium outside $A_1$
has a nonunique target projection. Thus empirical terminal samples
outside $A_1$ can occur only at target projection ties.

\subsection{Analytical contributions}\label{subsec:intro-contributions}

The proofs combine the geometry of nearest-point projections with
entropy estimates for the time-dependent dynamics.

\medskip
\noindent\textit{Signed geometry and dissipation.}
The separated projection envelope retains the nearest-point geometry
of each support, even when limiting directions depend on probability
masses and lie outside $\partial^C\Phi_\lambda$. Applying the chain
rule to each squared distance yields dissipation along the signed
limiting inclusion, which is the basis for equilibrium localization.

\medskip
\noindent\textit{Posterior entropy and deterministic action.}
Each logarithmic heat gradient is expressed through a posterior
mean associated with Gaussian smoothing. The relative entropy of this posterior controls the
time derivative of the regularized potential, giving energy convergence
and finite deterministic kinetic action in logarithmic time.
Compactness of translated paths then yields equilibrium localization.

\medskip
\noindent\textit{Conditional free energy and stochastic convergence.}
Applying It\^o's formula to the regularized energy
$F_\lambda(Y_\tau^\varepsilon,t_\tau)$ introduces the curvature term
$\varepsilon t_\tau\Delta_xF_\lambda(Y_\tau^\varepsilon,t_\tau)$.
Under extrapolation, the available upper bound for this term need not
be integrable in logarithmic time. We instead combine the expected
potential energy of a restarted trajectory with a
time-dependent multiple of the entropy of its transition density. After a short
smoothing interval for the point-mass initial law, this gives the
conditional estimate in \cref{lem:energy-estimates}\textup{(ii)}.
The bound is uniform in the restart time; its error decays with heat
time and grows at most quadratically in the restart state.
Moment bounds make these errors summable. An almost-supermartingale
argument then gives almost-sure energy convergence, and dissipation
along limiting paths yields equilibrium localization.

\medskip
\noindent\textit{Empirical rigidity and endpoints.}
For each nearest-point pattern, the equidistance and signed barycentric
constraints place equilibria in two affine spaces with orthogonal
direction spaces.
Their intersection therefore contains at most one equilibrium, even
for degenerate configurations. There are only finitely many patterns,
so the equilibrium set is finite without a genericity assumption.
Connectedness of the accumulation set then gives point convergence for
both dynamics, including projection-interface endpoints. At a strict-cell
endpoint, the posterior means are exponentially close to the selected
atoms. The resulting affine equation gives the deterministic
asymptotic expansion.

\subsection{Relation to existing work}\label{subsec:intro-related}

\medskip
\noindent\textit{Heat-kernel asymptotics and Hamilton--Jacobi equations.}
The small-time heat-kernel formula \cite{varadhan1967behavior} becomes
a support-distance Laplace principle in the convolution setting
\cite[Section~4.3]{DemboZeitouni1998}. We use its locally uniform form
for compactly supported, possibly singular measures. The logarithmic
heat transformation and squared-distance geometry also connect with
Hamilton--Jacobi and target-set eikonal equations
\cite{bardi1997optimal,crandall1992users,fleming2006controlled}.
The covariance identities in \cref{lem:heat-hessian} give the
Euclidean matrix estimate whose trace is the Li--Yau inequality
\cite{liyau1986parabolic}.
Convex mixing preserves the upper Hessian bound for the rescaled
potential; extrapolation need not preserve it.

\medskip
\noindent\textit{Nonsmooth geometry and gradient dynamics.}
Nearest-point sets and squared distances are treated by convex and
nonsmooth analysis
\cite{rockafellar1970convex,clarke1989optimization,rockafellar1998variational},
including the semiconcavity calculus of
\cite{cannarsa2004semiconcave}. The distinction between convergence of
energies and of their gradient evolutions is central to variational flow
theory \cite{ambrosio2008gradient,sandier2004gamma,serfaty2011gamma}.
Here the additional issue is that signed mixing can select directions
outside the limiting Clarke subdifferential. We therefore retain the two component subdifferentials when passing
to the limiting inclusion.

\medskip
\noindent\textit{Long-time dynamics and differential inclusions.}
Asymptotically autonomous evolution
\cite{markus1956asymptotically,galaktionov2004stability} and differential
inclusions \cite{aubin1990set} provide the long-time framework.
The trajectory estimates here combine posterior information in the
deterministic case with Fokker--Planck free energy
\cite{jordan1998variational} and almost-supermartingale convergence
\cite{robbins1971convergence} in the stochastic case.

\medskip
\noindent\textit{Score matching and heat-flow analysis.}
Score matching and diffusion models provide the generative setting
\cite{hyvarinen2005estimation,vincent2011connection,
sohl2015deep,ho2020denoising,song2021scorebased}.
The Gaussian posterior-mean representation is classical
\cite{fukunaga1975estimation,vincent2011connection}; stochastic sampling
also draws on reverse-diffusion theory
\cite{anderson1982reverse,haussmann1986reversal}.
The single-score analysis in \cite{liuzuazua2025} proves density stability
and concentration; its Theorem~3.7(i) establishes deterministic
compact-support endpoints for Lebesgue-almost every initialization.
Our signed-mixing analysis uses the component posteriors and the actual
transition densities.

\medskip
\noindent\textit{Empirical denoisers and trajectory endpoints.}
For a single empirical score, convergence at the forward-noise
standard-deviation scale under a strict-Voronoi-cell accumulation
hypothesis is proved in
\cite[Theorem~4.3 and Example~4.4]{baptista2025memorization}, which also exhibits an
exceptional interface endpoint. For flow-matching ODEs,
\cite[Theorems~5.3--5.4]{wan2025geometrydenoisers} establishes
almost-everywhere terminal convergence under geometric and local mass
assumptions, with refined rates for smooth submanifold-supported and
finite discrete laws. Our empirical theorem treats signed mixing
from every deterministic initialization and arbitrary finite site
configurations; the strict-cell assumption concerns the rate rather
than endpoint existence.

\medskip
\noindent\textit{Guidance and score composition.}
Guidance and score composition
\cite{dhariwal2021diffusion,ho2021classifierfree,liu2022composable,du2023reduce}
have been analyzed through boundary concentration and score-error
sensitivity \cite{chidambaram2024does}, predictor--corrector structure
\cite{bradley2025predictor}, and stagewise attraction and absorption
\cite{cai2026particle}. For genuine conditional guidance,
\cite{azangulov2026adaptive} proves almost-sure terminal convergence on
the conditional support under compactness and bounded guidance strength.
Support robustness, including exponential-integrator discretizations,
is established in \cite{cao2026robustness} for separated compact convex
mixture supports with a positive target-density lower bound.
Our focus is the terminal dynamics of fixed signed mixtures of possibly
unrelated laws. We identify a dissipative nearest-support inclusion even
when limiting directions lie outside $\partial^C\Phi_\lambda$, prove
endpoint convergence for arbitrary finite configurations, and obtain
deterministic strict-cell asymptotics. The nested-support corollary
restricts off-target equilibria to target projection ties under support
inclusion alone. These
results complement conditional-support guarantees and distributional
approximation estimates such as
\cite{debortoli2022manifold,conforti2025kl}.

\subsection{Organization}\label{subsec:intro-organization}
Section~\ref{sec:main-results} presents the geometric limits and
deterministic asymptotics. Section~\ref{sec:perturbation-stochastic}
treats perturbations and stochastic asymptotics, and
Section~\ref{sec:numerics} gives numerical illustrations.
The main results are proved in Section~\ref{sec:proofs}, using five
lemmas whose proofs appear in Section~\ref{sec:technical-proofs}.
Section~\ref{sec:conclusion} discusses further questions.

\section{Geometric limits and deterministic asymptotics}\label{sec:main-results}

We write $\Pcal(\R^d)$ for Borel probability measures,
$\langle\cdot,\cdot\rangle$ for the Euclidean scalar product, and
$|\cdot|$ for the associated norm. The parameter $\lambda\geq0$ is
fixed. Terms with zero coefficients are omitted
throughout, including in set-valued expressions: $\lambda=0$ uses only
the second law, and $\lambda=1$ only the first. An \emph{active support}
means a support with nonzero coefficient. Compactness hypotheses are
specified in the theorem statements and concern only these supports.

The heat regularizations introduced in \cref{eq:intro-heat} are
\begin{equation}\label{eq:heat-flow}
 u_k(x,t)=\int_{\R^d}G_t(x-a)\,du_{0,k}(a),\qquad
 G_t(z)=(4\pi t)^{-d/2}\exp\!\left(-\frac{|z|^2}{4t}\right).
\end{equation}
We recall the regularized and limiting potentials from
\eqref{eq:intro-Phi}:
\begin{equation}\label{eq:framework-potentials}
 \begin{aligned}
 F_\lambda(x,t)
 &=-4t\left(\lambda\log u_1(x,t)+(1-\lambda)\log u_2(x,t)\right),\\
 \Phi_\lambda(x)
 &=\lambda\dist(x,A_1)^2+(1-\lambda)\dist(x,A_2)^2,
 \qquad A_k=\supp u_{0,k}.
 \end{aligned}
\end{equation}
The relation $\nabla F_\lambda=-4t\,s^{(\lambda)}$ expresses the
mixed score through the regularized potential.

Write $d_k(x)=\dist(x,A_k)$ and
$\Pi_k(x)=\argmin_{a\in A_k}|x-a|$. Define the nearest-support
equilibrium set and the confinement set by
\begin{equation}\label{eq:geometric-sets}
 \begin{aligned}
 \Ecal_\lambda
 &=\bigl\{x\in\R^d:
 x\in\lambda\conv\Pi_1(x)+(1-\lambda)\conv\Pi_2(x)\bigr\},\\
 K_\lambda&=\lambda\conv(A_1)+(1-\lambda)\conv(A_2).
 \end{aligned}
\end{equation}
The set operations are Minkowski addition and scalar multiplication,
including negative coefficients. At projection ties, the convex hulls
allow combinations of nearest points.
When the active supports are compact, $K_\lambda$ is compact and convex,
and $\Ecal_\lambda\subseteq K_\lambda$ is compact. Its points are
\emph{nearest-support equilibria}: the state equals a weighted combination
of nearest-point barycentres of the active supports. The equilibrium set can contain nonminimizing points, while the
dynamics determines which equilibria occur as trajectory endpoints.

For a locally Lipschitz function $h$, the Clarke subdifferential is
\begin{equation}\label{eq:Clarke-definition}
 \partial^C h(x)=\conv\left\{\lim_{n\to\infty}\nabla h(x_n):
 x_n\to x,\ h\text{ differentiable at }x_n\right\}.
\end{equation}
The squared-distance formula recalled in
\cref{lem:distance-geometry}\textup{(i)} yields
\begin{equation}\label{eq:outer-Clarke}
 \begin{aligned}
 \widehat\partial\Phi_\lambda(x)
 &:=\lambda\partial^C d_1^2(x)+(1-\lambda)\partial^C d_2^2(x)\\
 &=2\bigl(x-\lambda\conv\Pi_1(x)-(1-\lambda)\conv\Pi_2(x)\bigr).
 \end{aligned}
\end{equation}
We call this the \emph{separated projection envelope}. Thus
$x\in\Ecal_\lambda$ if and only if $0\in\widehat\partial\Phi_\lambda(x)$.
The envelope is determined entirely by the supports and $\lambda$. For a set-valued map $\mathcal F$, write
$\gph\mathcal F=\{(x,p):p\in\mathcal F(x)\}$, with graph distance measured
in $\R^d\times\R^d$.

The logarithmic time change
\begin{equation}\label{eq:similarity-time}
 \tau=\log(T/t),\qquad t_\tau=Te^{-\tau},\qquad Y_\tau=X_{t_\tau}
\end{equation}
rewrites \cref{eq:physical-flow} as
\begin{equation}\label{eq:similarity-flow-general}
 \dot Y_\tau=-\frac14\nabla F_\lambda(Y_\tau,t_\tau),
 \qquad Y_0=y_0,\quad \tau\geq0.
\end{equation}
Here $y_0=x_T$, and the inverse relation is $X_t=Y_{\log(T/t)}$.
The processes $X$ and $Y$ have the same geometric path and accumulation
points. Equal increments of $\tau$ correspond to equal multiplicative
reductions in heat time. The time change absorbs the explicit
inverse-time factor, although spatial derivatives may still grow near
projection interfaces.

\subsection{Geometric reduction and autonomous limiting dynamics}
\label{subsec:geometric-result}

The first theorem identifies the limiting score geometry and the
equations satisfied by limits of translated trajectories. Graphical
convergence allows the spatial point to vary with heat time and thus
captures distinct limiting directions at projection ties.
A solution of a differential inclusion below means a locally absolutely
continuous curve satisfying the inclusion almost everywhere.

\begin{theorem}[Geometric reduction]\label[theorem]{thm:support-reduction}
Let $\lambda\geq0$ and assume that the active supports are compact.
\begin{enumerate}[label=\textup{(\roman*)}]
\item The potentials $F_\lambda(\cdot,t)$ converge locally uniformly to
$\Phi_\lambda$ as $t\to0^+$. If $x_n\to x$, $t_n\to0^+$, and
$\nabla F_\lambda(x_n,t_n)\to p$, then
$p\in\widehat\partial\Phi_\lambda(x)$. For every nonempty compact $B\subset\R^d$,
\begin{equation}\label{eq:compact-graph-convergence}
 \sup_{x\in B}\dist\!\left(
 \left(x,\nabla F_\lambda(x,t)\right),
 \gph\widehat\partial\Phi_\lambda\right)\longrightarrow0.
\end{equation}
If $0\leq\lambda\leq1$, then
$\widehat\partial\Phi_\lambda=\partial^C\Phi_\lambda$ and the graph
limit is complete: for every $x\in\R^d$,
$p\in\partial^C\Phi_\lambda(x)$, and sequence $t_n\to0^+$,
there exists a sequence $x_n'\to x$ such that
\begin{equation}\label{eq:convex-graph-recovery}
 \nabla F_\lambda(x_n',t_n)\longrightarrow p.
\end{equation}

\item For every $T>0$ and $y_0\in\R^d$, the solution $Y$ of
\cref{eq:similarity-flow-general} is unique, global, and precompact.
For any sequence $\tau_n\to\infty$, the translates
$Y_{\tau_n+\cdot}$ are relatively compact in
$C_{\mathrm{loc}}\left([0,\infty);\R^d\right)$.
Every limit $Z$ takes values in $K_\lambda$ and satisfies
\begin{equation}\label{eq:limit-inclusion-main}
 \dot Z_\tau\in-\frac14\widehat\partial\Phi_\lambda(Z_\tau)
 \qquad\text{for a.e. }\tau\geq0.
\end{equation}
Every solution of this inclusion obeys, for $\tau\leq\tau'$ in its domain,
\begin{equation}\label{eq:profile-energy}
 \Phi_\lambda(Z_{\tau'})+4\int_\tau^{\tau'}|\dot Z_s|^2\,ds
 =\Phi_\lambda(Z_\tau).
\end{equation}
\end{enumerate}
\end{theorem}

\begin{proof}
See Section~\ref{subsec:geometry-proof}.
\end{proof}

\begin{remark}[The original-time deterministic process]
\label[remark]{rem:geometric-physical}
Under \cref{thm:support-reduction}, $X_t:=Y_{\log(T/t)}$ is the unique
solution of \eqref{eq:physical-flow} on $(0,T]$, with $x_T=y_0$.
Its range and accumulation set as $t\to0^+$ coincide with those of $Y$
as $\tau\to\infty$; in particular, $X$ is bounded.
\end{remark}

For convex mixing the autonomous inclusion is the Clarke subgradient
system of $\Phi_\lambda$. For extrapolation it is generally an outer
geometric description. Attainable field directions can form a strict subset of
the envelope, and the inclusion describes a potentially larger class
of trajectories.
If all active projections are unique on an open set, the field converges
to $\nabla\Phi_\lambda$ uniformly on compact subsets of that set.

\subsection{Omega-limit localization for compact supports}
\label{subsec:compact-result}

We next prove energy convergence and show that every translated limit
is stationary.
For a precompact continuous curve, define
\begin{equation}\label{eq:omega-definition}
 \omega(Y)=\bigcap_{S>0}\overline{Y\left([S,\infty)\right)}.
\end{equation}
This is the set of all accumulation points as $\tau\to\infty$.
We say that a law $u_0$ with support $A$ satisfies a \emph{uniform lower
small-ball mass condition} if there exist $c>0$, $\alpha\geq0$, and
$r_0>0$ such that
\begin{equation}\label{eq:lower-mass}
 \inf_{a\in A}u_0\left(B(a,r)\right)\geq cr^\alpha,
 \qquad 0<r\leq r_0.
\end{equation}
This condition is expressed in terms of small-ball probabilities and
also applies to singular measures. Finite empirical laws with positive
weights satisfy it with $\alpha=0$. A density bounded
below on a compact support with a uniform interior-volume bound satisfies
it with $\alpha=d$.

\Needspace{8\baselineskip}
\begin{theorem}[Omega-limit localization]\label[theorem]{thm:compact-omega}
Let $\lambda\geq0$ and assume that the active supports are compact.
If $\lambda>1$, assume in addition that $u_{0,2}$ satisfies
\cref{eq:lower-mass}. For every $T>0$ and $y_0\in\R^d$, let $Y$ be the
solution of \cref{eq:similarity-flow-general}.
\begin{enumerate}[label=\textup{(\roman*)}]
\item The trajectory has finite kinetic action,
\begin{equation}\label{eq:compact-finite-action}
 \int_0^\infty|\dot Y_\tau|^2\,d\tau<\infty,
\end{equation}
and there exists $\ell\in\R$ such that
\begin{equation}\label{eq:common-energy-limit}
 F_\lambda(Y_\tau,t_\tau)\longrightarrow\ell,
 \qquad \Phi_\lambda(Y_\tau)\longrightarrow\ell.
\end{equation}
\item The set $\omega(Y)$ is nonempty, compact, and connected, with
\begin{equation}\label{eq:compact-omega-localization}
 \omega(Y)\subseteq
 \{x\in\Ecal_\lambda:\Phi_\lambda(x)=\ell\}.
\end{equation}
\end{enumerate}
\end{theorem}

\begin{proof}
See Section~\ref{subsec:compact-proof}.
\end{proof}

\begin{remark}[Localization in the original heat time]
\label[remark]{rem:compact-physical}
For $X_t=Y_{\log(T/t)}$, \cref{thm:compact-omega} gives
\[
 \int_0^T t\,|\dot X_t|^2\,dt
 =\int_0^\infty|\dot Y_\tau|^2\,d\tau<\infty.
\]
As $t\to0^+$, both $F_\lambda(X_t,t)$ and $\Phi_\lambda(X_t)$ converge
to $\ell$, and all accumulation points of $X$ belong to
$\Ecal_\lambda\cap\{\Phi_\lambda=\ell\}$.
\end{remark}

For extrapolation, the additional mass hypothesis is imposed only on the negatively weighted
reference law; convex mixing, including either single-score case,
requires only compactness.

Finite action makes all translated limits stationary, as quantified by
\cref{eq:stationary-translates} in the proof. It controls motion on
fixed time windows, while the total trajectory length may be infinite.
For empirical data, finiteness of $\Ecal_\lambda$ and connectedness of
$\omega(Y)$ yield convergence to a single point, as stated below.

\subsection{Empirical endpoint convergence and asymptotics}
\label{subsec:empirical-result}

For finite atomic measures, called empirical measures below, the
compact-support result strengthens to convergence of each trajectory. Assume
\begin{equation}\label{eq:empirical-measures}
 u_{0,k}=\sum_{i=1}^{n_k}w_{k,i}\delta_{a_{k,i}},\qquad
 w_{k,i}>0,\qquad\sum_{i=1}^{n_k}w_{k,i}=1,
\end{equation}
with distinct sites within each support. Set
\begin{equation}\label{eq:active-skeleton}
 \begin{gathered}
 I_k(x)=\argmin_i|x-a_{k,i}|,\qquad
 \Scal_k=\{x:\#I_k(x)>1\},\quad k=1,2,\\
 \Scal_\lambda=
 \begin{cases}
  \Scal_2,&\lambda=0,\\
  \Scal_1,&\lambda=1,\\
  \Scal_1\cup\Scal_2,&\lambda>0,\ \lambda\neq1.
 \end{cases}
 \end{gathered}
\end{equation}
Outside the active Voronoi skeleton $\Scal_\lambda$, every active support
has a unique nearest atom. We call an endpoint outside this set a
\emph{strict-cell endpoint}. The equilibrium set $\Ecal_\lambda$ itself
also includes possible ties.

\Needspace{8\baselineskip}
\begin{theorem}[Empirical convergence and strict-cell asymptotics]
\label[theorem]{thm:empirical-main}
Assume \cref{eq:empirical-measures}, and fix $\lambda\geq0$ and $T>0$.
For every initial state $y_0\in\R^d$, the solution $Y$ of
\cref{eq:similarity-flow-general} is unique on $[0,\infty)$ and satisfies:
\begin{enumerate}[label=\textup{(\roman*)}]
\item The set $\Ecal_\lambda$ is finite. There exists $y^*\in\Ecal_\lambda$
such that
\begin{equation}\label{eq:empirical-convergence}
 Y_\tau\longrightarrow y^*\qquad(\tau\to\infty).
\end{equation}
\item If $y^*\notin\Scal_\lambda$, there exist
$\xi^*\in\R^d$ and $\eta>0$ such that
\begin{equation}\label{eq:tau-sharp-rate}
 Y_\tau=y^*+e^{-\tau/2}\xi^*
       +O\!\left(t_\tau e^{-\eta/t_\tau}\right)
 \qquad(\tau\to\infty).
\end{equation}
Consequently, there exists $C>0$ such that
\begin{equation}\label{eq:strict-cell-rate}
 |Y_\tau-y^*|\leq Ce^{-\tau/2},\qquad \tau\geq0.
\end{equation}
\end{enumerate}
\end{theorem}

\begin{proof}
See Section~\ref{subsec:empirical-proof}.
\end{proof}

\begin{remark}[Original-time form and sharpness]
\label[remark]{rem:sharp-asymptotics}
The original-time solution $X_t=Y_{\log(T/t)}$, with $x_T=y_0$, satisfies
\begin{equation}\label{eq:empirical-physical-convergence}
 X_t\longrightarrow x^*:=y^*\qquad(t\to0^+).
\end{equation}
If $y^*\notin\Scal_\lambda$, \eqref{eq:strict-cell-rate} gives
$|X_t-x^*|\leq C\sqrt{t/T}$ for every $0<t\leq T$.
Let $a_{k,*}$ be the unique nearest atom at this endpoint in each active
support. Then $x^*=y^*=\lambda a_{1,*}+(1-\lambda)a_{2,*}$, and the
expansion in \cref{thm:empirical-main} becomes
\begin{equation}\label{eq:t-sharp-rate}
 X_t=x^*+\sqrt{t/T}\,\xi^*
       +O\!\left(te^{-\eta/t}\right)
 \qquad(t\to0^+).
\end{equation}
Thus the heat-time exponent $1/2$ is sharp when $\xi^*\neq0$;
if $\xi^*=0$, then $|X_t-x^*|=o(t^m)$ for every $m>0$.
The constants may depend on the data,
weights, initialization, $T$, and $\lambda$, and may deteriorate as
the Voronoi gap approaches zero.
\end{remark}

For finite empirical laws, positivity of atomic weights implies \cref{eq:lower-mass}
directly. The strict-cell hypothesis enters the rate estimate, whereas
endpoint convergence holds for arbitrary empirical configurations.

\begin{remark}[Limits of the conclusions]\label[remark]{rem:counterexamples}
In one dimension, let
$u_{0,1}=(\delta_{-1}+3\delta_1)/4$,
$u_{0,2}=(3\delta_{-1}+\delta_1)/4$, and $\lambda=2$.
Both supports are $\{-1,1\}$, so $\Phi_2(x)=(1-|x|)^2$, but
\[
 \partial^C\Phi_2(0)=[-2,2],\qquad
 \widehat\partial\Phi_2(0)=[-6,6],\qquad
 \nabla F_2(0,t)=-3.
\]
Thus the limiting direction at this interface lies outside the Clarke
subdifferential of the limiting potential. If instead
$u_{0,1}=u_{0,2}=(\delta_{-1}+\delta_1)/2$, then
$\nabla F_\lambda(x,t)=2\{x-\tanh(x/(2t))\}$ for every $\lambda$.
Its joint limits as $x\to0$ and $t\to0^+$ form $[-2,2]$,
strictly smaller than $\widehat\partial\Phi_2(0)$.
Moreover, the trajectory initialized at zero stays there, although
zero is a strict local maximum of $\Phi_\lambda(x)=(1-|x|)^2$.
Consequently, convergence to an equilibrium cannot be strengthened
to convergence to a local minimum for every initialization.
\end{remark}

\section{Perturbations and stochastic asymptotics}
\label{sec:perturbation-stochastic}

We extend the convergence results to square-integrable deterministic
perturbations and to stochastic sampling. As before, the results are
stated in logarithmic time and then translated into heat time.

\subsection{Perturbations of the sampling scores}

Consider the perturbed dynamics
\begin{equation}\label{eq:perturbed-flow}
 \dot{\widetilde Y}_\tau=-\frac14\nabla F_\lambda(\widetilde Y_\tau,t_\tau)+r_\tau,
 \qquad \widetilde Y_0=\widetilde y_0.
\end{equation}
Here $r$ is a prescribed measurable error in logarithmic time.

\begin{proposition}[Robustness of equilibrium localization]
\label[proposition]{prop:L2-robustness}
Assume the hypotheses on $\lambda$ and the data laws in
\cref{thm:compact-omega}. Fix $T>0$, $\widetilde y_0\in\R^d$, and
$r\in L^2(0,\infty;\R^d)$. Equation~\eqref{eq:perturbed-flow} has a
unique global bounded Carath\'eodory solution $\widetilde Y$ with the
following properties.
\begin{enumerate}[label=\textup{(\roman*)}]
\item The trajectory has finite kinetic action,
\begin{equation}\label{eq:perturbed-finite-action}
 \int_0^\infty|\dot{\widetilde Y}_\tau|^2\,d\tau<\infty,
\end{equation}
and there exists $\widetilde\ell\in\R$ such that
\begin{equation}\label{eq:perturbed-energy-limits}
 F_\lambda(\widetilde Y_\tau,t_\tau)\longrightarrow\widetilde\ell,
 \qquad \Phi_\lambda(\widetilde Y_\tau)\longrightarrow\widetilde\ell.
\end{equation}
\item The set $\omega(\widetilde Y)$ is nonempty, compact, and connected,
with
\begin{equation}\label{eq:perturbed-omega}
 \omega(\widetilde Y)
 \subseteq\{x\in\Ecal_\lambda:\Phi_\lambda(x)=\widetilde\ell\}.
\end{equation}
\end{enumerate}
\end{proposition}

\begin{proof}
See Section~\ref{subsec:robustness-proof}.
\end{proof}

\begin{remark}[Errors in the original score]
\label[remark]{rem:physical-score-robustness}
Under the hypotheses on $\lambda$ and the data in \cref{prop:L2-robustness},
let $\widetilde X$ be a locally absolutely continuous solution of
$\dot{\widetilde X}_t=-s^{(\lambda)}(\widetilde X_t,t)
-e(\widetilde X_t,t)$ on $(0,T]$. If
\begin{equation}\label{eq:physical-L2-error}
 \int_0^T t\,|e(\widetilde X_t,t)|^2\,dt<\infty,
\end{equation}
then $\widetilde X$ is bounded and its accumulation points as $t\to0^+$
belong to $\Ecal_\lambda$. For finite empirical data, finiteness of
$\Ecal_\lambda$ and connectedness of $\omega(\widetilde Y)$ give
$\widetilde Y_\tau\to\widetilde y^*\in\Ecal_\lambda$, hence
$\widetilde X_t\to\widetilde x^*:=\widetilde y^*$, where
$\widetilde Y_\tau=\widetilde X_{t_\tau}$ and
$\widetilde Y_0=\widetilde X_T$.
Indeed, the change of time \eqref{eq:similarity-time} gives
$r_\tau=t_\tau e(\widetilde X_{t_\tau},t_\tau)$ and
$\int_0^\infty|r_\tau|^2\,d\tau
=\int_0^T t\,|e(\widetilde X_t,t)|^2\,dt$, so
\cref{prop:L2-robustness} applies.
Thus \eqref{eq:physical-L2-error} is a sufficient condition for
localization and empirical endpoint convergence under score approximation
along the given trajectory. The selected equilibrium may change, and even an unchanged
endpoint may be approached more slowly. For example, take
$u_{0,1}=u_{0,2}=\delta_c$, $r_\tau=e^{-\tau/4}v$ with $v\neq0$, and
$\widetilde Y_0=c$. Then
\[
 \widetilde Y_\tau-c=4(e^{-\tau/4}-e^{-\tau/2})v,
\]
so the endpoint remains $c$, but the leading error has scale
$e^{-\tau/4}$, or $(t/T)^{1/4}$ in heat time.
Deriving \eqref{eq:physical-L2-error} from a training loss or a
time-discretization error requires additional analysis.
\end{remark}

\subsection{Stochastic localization and endpoint convergence}
\label{subsec:stochastic-generation}

We use the same decreasing heat time as in the deterministic process.
For $\varepsilon\geq0$, consider the stochastic equation
\begin{equation}\label{eq:backward-generation-sde}
 dX_t^\varepsilon
 =-(1+\varepsilon)s^{(\lambda)}(X_t^\varepsilon,t)\,dt
       +\sqrt{2\varepsilon}\,dW_t,
 \qquad X_T^\varepsilon=x_T,\quad 0<t\leq T.
\end{equation}
On a probability space $(\Omega,\mathcal F,\mathbb P)$, let $W_T=0$
and let $W$ be Brownian motion for a backward filtration $(\mathcal F_t)$
with the usual conditions in decreasing time.
For $0<t<t'\leq T$, the increment $W_t-W_{t'}$ is independent of
$\mathcal F_{t'}$ and is centered Gaussian with covariance $(t'-t)I_d$.
Thus \eqref{eq:backward-generation-sde} means
\[
 X_t^\varepsilon=x_T+(1+\varepsilon)\int_t^T
       s^{(\lambda)}(X_{t'}^\varepsilon,t')\,dt'
       +\sqrt{2\varepsilon}\,W_t.
\]
A strong solution is an $(\mathcal F_t)$-adapted process satisfying
this identity almost surely for all $t\in(0,T]$, with sample paths in
$C((0,T];\R^d)$. Uniqueness is pathwise for a prescribed Brownian
motion and terminal state. For a single score, $\varepsilon=1$ is the
classical reverse diffusion when initialized with the terminal heat law
\cite{anderson1982reverse,haussmann1986reversal}.

The change of time \eqref{eq:similarity-time} gives
\begin{equation}\label{eq:similarity-sde}
 dY_\tau^\varepsilon
 =-\frac{1+\varepsilon}{4}
       \nabla F_\lambda(Y_\tau^\varepsilon,t_\tau)\,d\tau
       +\sqrt{2\varepsilon t_\tau}\,dW_\tau,
 \qquad Y_0^\varepsilon=y_0=x_T.
\end{equation}
We reuse $W$ for the standard Brownian motion in logarithmic time,
adapted to the increasing filtration
$\mathcal G_\tau:=\mathcal F_{t_\tau}$
\cite{karatzas1991brownian,revuzyor1999continuous}.
Strong solutions of \eqref{eq:similarity-sde} are understood
in integral form and have continuous paths on $[0,\infty)$.

For a stochastic process, $\omega(Y^\varepsilon)$ denotes the ordinary
omega-limit set \eqref{eq:omega-definition} of each continuous sample
path. An $\Ecal_\lambda$-valued random vector belongs to $\Ecal_\lambda$
almost surely.

\Needspace{9\baselineskip}
\begin{theorem}[Stochastic localization and empirical convergence]
\label[theorem]{thm:stochastic-convergence}
Let $\lambda\geq0$ and assume that the active supports are compact.
If $\lambda>1$, assume that $u_{0,2}$ satisfies \eqref{eq:lower-mass}.
Fix $T>0$, $y_0\in\R^d$, and $\varepsilon\geq0$.
Equation~\eqref{eq:similarity-sde} has a unique global strong solution
$Y^\varepsilon$ with continuous, almost surely bounded paths.
\begin{enumerate}[label=\textup{(\roman*)}]
\item There is a finite random variable $\ell^\varepsilon$ such that,
almost surely,
\begin{gather}
 F_\lambda(Y_\tau^\varepsilon,t_\tau)\longrightarrow\ell^\varepsilon,
 \qquad
 \Phi_\lambda(Y_\tau^\varepsilon)\longrightarrow\ell^\varepsilon
 \qquad(\tau\to\infty),
 \label{eq:stochastic-energy-limits}\\
 \varnothing\neq\omega(Y^\varepsilon)
 \subseteq\{x\in\Ecal_\lambda:\Phi_\lambda(x)=\ell^\varepsilon\}.
 \label{eq:stochastic-omega}
\end{gather}
The omega-limit set is compact and connected. In particular,
$\dist(Y_\tau^\varepsilon,\Ecal_\lambda)\to0$ as $\tau\to\infty$.
\item If the active laws are finite empirical measures, there is an
$\Ecal_\lambda$-valued random vector $y^{\varepsilon,*}$ such that
\begin{equation}\label{eq:stochastic-endpoint}
 Y_\tau^\varepsilon\longrightarrow y^{\varepsilon,*}
 \qquad(\tau\to\infty),\qquad\text{almost surely}.
\end{equation}
\end{enumerate}
\end{theorem}

\begin{proof}
See Section~\ref{subsec:stochastic-proof}.
\end{proof}

\begin{remark}[The original-time stochastic process]
\label[remark]{rem:stochastic-physical}
The process $X_t^\varepsilon:=Y_{\log(T/t)}^\varepsilon$ solves
\eqref{eq:backward-generation-sde}, with $x_T=y_0$, and has the same
accumulation points as $Y^\varepsilon$. Thus the energy and localization
conclusions hold as $t\to0^+$ as well. For empirical data, setting
$x^{\varepsilon,*}:=y^{\varepsilon,*}$ gives
\begin{equation}\label{eq:stochastic-physical-endpoint}
 X_t^\varepsilon\longrightarrow x^{\varepsilon,*}\in\Ecal_\lambda
 \quad(t\to0^+),\qquad\text{almost surely}.
\end{equation}
For each fixed initialization, assigning this value at $t=0$ gives
an almost surely continuous extension to $[0,T]$.
For general compact laws, the accumulation set is localized to a single
equilibrium energy level.
In conditional guidance with nested supports, \cref{rem:nested}
further restricts the possible terminal locations.
\end{remark}

\begin{corollary}[Terminal sampling law]
\label[corollary]{cor:terminal-law}
Assume that the active laws are finite atomic measures, and fix
$\lambda\geq0$, $T>0$, and $\varepsilon\geq0$.
Let $X_T^\varepsilon$ have any law $\mu_T\in\Pcal(\R^d)$,
independently of the driving Brownian motion when $\varepsilon>0$.
There is a probability measure $\mu_0^\varepsilon$ such that,
narrowly as $t\to0^+$,
\[
 \operatorname{Law}(X_t^\varepsilon)
 \rightharpoonup\mu_0^\varepsilon,
 \qquad \supp\mu_0^\varepsilon\subseteq\Ecal_\lambda.
\]
In particular, the terminal law is finitely supported.
\end{corollary}

\begin{proof}
See Section~\ref{subsec:stochastic-proof}.
\end{proof}

\begin{remark}[Square-integrable perturbations for empirical data]
\label[remark]{rem:stochastic-perturbations}
For finite empirical active laws, the endpoint conclusion remains valid
with an additional prescribed drift $r_\tau\,d\tau$ in
\eqref{eq:similarity-sde}, progressively measurable for
$(\mathcal G_\tau)$, provided
\begin{equation}\label{eq:stochastic-L2-error}
 \int_0^\infty|r_\tau|^2\,d\tau<\infty
 \qquad\text{almost surely}.
\end{equation}
For a well-posed sampler using $s^{(\lambda)}+e$, this is
\eqref{eq:physical-L2-error} almost surely along its path, since
$r_\tau=(1+\varepsilon)t_\tau
 e(\widetilde X_{t_\tau}^\varepsilon,t_\tau)$ and
\begin{equation}\label{eq:stochastic-physical-error}
 \int_0^\infty|r_\tau|^2\,d\tau
 =(1+\varepsilon)^2\int_0^T
       t\,|e(\widetilde X_t^\varepsilon,t)|^2\,dt.
\end{equation}
This sufficient condition concerns the error along the perturbed
trajectory; it preserves endpoint convergence, but need not preserve
the selected equilibrium or its law.
\end{remark}

\begin{proof}
See Section~\ref{subsec:stochastic-proof}.
\end{proof}

In classifier-free guidance \cite{ho2021classifierfree}, a positive-weight
conditional component of the reference law gives $A_1\subseteq A_2$,
leading to the following geometric restriction.

\begin{corollary}[Conditional-support geometry]
\label[corollary]{rem:nested}
Assume that $A_1\subseteq A_2$ are compact and $\lambda>1$.
The local and global minimizers of $\Phi_\lambda$ are exactly $A_1$,
and every equilibrium outside $A_1$ has a nonunique projection onto $A_1$.
For finite atomic laws, the endpoints belong to $A_1\cup\Scal_1$
for every deterministic initialization and almost surely for each fixed
stochastic initialization.
\end{corollary}

\begin{proof}
Since $d_2\leq d_1$, we have $\Phi_\lambda\geq d_1^2$, with value zero
on $A_1$. For squared distances, the one-sided directional derivative is
$(d_k^2)'(x;v)=2\min_{b\in\Pi_k(x)}\langle x-b,v\rangle$.
Thus, for $x\notin A_1$ and $a\in\Pi_1(x)$, Cauchy--Schwarz gives
\[
 \Phi_\lambda'(x;a-x)
 \leq-2\lambda d_1(x)^2+2(\lambda-1)d_1(x)d_2(x)
 \leq-2d_1(x)^2<0.
\]
Thus the local and global minimizers are exactly $A_1$.
If $x\in\Ecal_\lambda$ has the unique target projection $a$, write
$x=\lambda a+(1-\lambda)b$ with $b\in\conv\Pi_2(x)$.
Then $\lambda(x-a)=(\lambda-1)(x-b)$ and $|x-b|\leq d_2(x)$, so
\[
 \frac{\lambda}{\lambda-1}d_1(x)\leq d_2(x)\leq d_1(x).
\]
It follows that $x\in A_1$. The endpoint assertion follows from
\cref{thm:empirical-main,thm:stochastic-convergence}.
\end{proof}

For empirical conditional guidance, a terminal sample outside the target
data set $A_1$ can occur only where several target atoms tie as nearest
points. A deterministic trajectory ending at a target atom has a unique
nearest point in each support, so \cref{rem:sharp-asymptotics} gives
an $O(\sqrt t)$ stopping error.
These conclusions give a geometric restriction under support inclusion
alone and complement the conditional-support guarantees in
\cite{azangulov2026adaptive,cao2026robustness}.

\section{Numerical illustrations of the convergence theory}
\label{sec:numerics}

Two examples illustrate the convergence results: continuous compact
supports, and finite atomic supports with deterministic and stochastic
endpoints. All computations use exact heat-flow scores and integrate in
logarithmic time from $T=1$ to positive stopping times. For continuous supports, we evaluate
the Gaussian integrals and use the Radau solver with relative and absolute
tolerances $10^{-10}$ and $10^{-12}$. For the empirical data, posterior
weights are evaluated in a numerically stabilized form; deterministic
trajectories use Radau, and stochastic trajectories use an exponential
step for the linear drift and noise, with implicit treatment of the
reference posterior when $\lambda>1$. The stochastic curves use $2048$
steps per decade of heat time and were compared with $512$ and $1024$
steps per decade using coupled noise increments.

\subsection{Continuous supports: geometry and localization}
\label{subsec:continuous-experiment}

Consider the segmented supports
\[
 I=[-1.5,-1.2]\cup[-0.8,0.8]\cup[1.2,1.5],\qquad
 A_1=I\times\{-1,1\},\qquad A_2=\{-1,1\}\times I,
\]
equipped with normalized arclength measures. These measures satisfy
the lower small-ball mass condition \eqref{eq:lower-mass}, so
\cref{thm:support-reduction,thm:compact-omega} apply in both
the convex and extrapolative regimes.

Figure~\ref{fig:continuous-data} superposes trajectories of the
time-dependent equation \eqref{eq:similarity-flow-general} on the
level sets of its limiting potential $\Phi_\lambda$.
The local-minimum sets are
\[
 \mathcal M_{1/2}=\{-1.1,-0.9,0.9,1.1\}^{2},\qquad
 \mathcal M_2=\{-2,0,2\}\times\{-1,1\}.
\]
Convex mixing balances proximity to the two supports, whereas
extrapolation produces four minima outside their joint convex hull.
Initial states are chosen so that the trajectories approach every
marked minimum, illustrating the dependence of the endpoint on the
initialization.
This behavior is consistent with equilibrium localization in
\cref{thm:compact-omega}; the theorem also permits nonminimizing
equilibria, as illustrated by \cref{rem:counterexamples}.

\begin{figure}[!htbp]
\centering
\begin{minipage}{0.48\textwidth}\centering
\includegraphics[width=\linewidth]{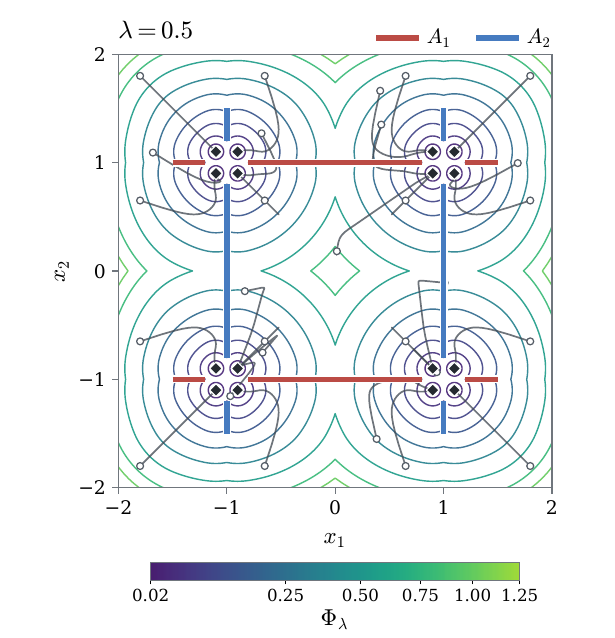}
\end{minipage}\hfill
\begin{minipage}{0.48\textwidth}\centering
\includegraphics[width=\linewidth]{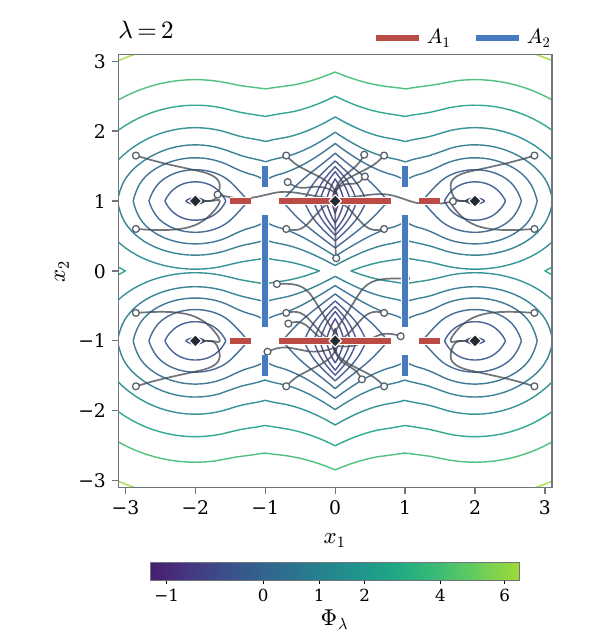}
\end{minipage}
\caption{Continuous-support geometry and deterministic trajectories.
Red and blue segments represent $A_1$ and $A_2$; colored contours
are level sets of $\Phi_\lambda$. Gray curves show $28$ trajectories
from the open-circle initializations to heat time $10^{-6}$.
Black diamonds mark local minima. Left: convex mixing,
$\lambda=0.5$. Right: extrapolation, $\lambda=2$.}
\label{fig:continuous-data}
\end{figure}

Table~\ref{tab:grid-defects} compares $F_\lambda$ and its gradient
with their limiting geometric counterparts on the fixed grid
\[
 \mathcal G=\{-1.8,-1.6,\ldots,1.8\}^2.
\]
The decreasing defects illustrate the potential convergence and
field localization of \cref{thm:support-reduction}.
The field is compared with the separated envelope
$\widehat\partial\Phi_\lambda$, which coincides with
$\partial^C\Phi_\lambda$ for convex mixing but may be larger
under extrapolation. These fixed-grid diagnostics complement
the moving-point graphical formulation of the theorem.

\begin{table}[!htbp]
\centering
\caption{Maximum potential and field defects on the fixed
$19\times19$ grid $\mathcal G$.}
\label{tab:grid-defects}
\begin{tabular}{ccccc}\toprule
&\multicolumn{2}{c}{$|F_\lambda-\Phi_\lambda|$}
&\multicolumn{2}{c}{$\dist(\nabla F_\lambda,
                           \widehat\partial\Phi_\lambda)$}\\
\cmidrule(lr){2-3}\cmidrule(lr){4-5}
$t$&$\lambda=0.5$&$\lambda=2$&$\lambda=0.5$&$\lambda=2$\\\midrule
$10^{-1}$&1.2478&1.7735&0.4104&1.2947\\
$10^{-2}$&0.0909&0.1640&0.1559&0.4931\\
$10^{-3}$&0.0085&0.0199&0.0505&0.1596\\
$10^{-4}$&0.0010&0.0024&0.0160&0.0505\\\bottomrule
\end{tabular}
\end{table}
\FloatBarrier

\subsection{Empirical supports: endpoints and convergence scales}
\label{subsec:empirical-2d-experiment}

Consider uniform empirical laws on
\begin{equation}\label{eq:two-dimensional-supports}
 \begin{aligned}
 A_1&=\{(-3.0,2.2),(-1.2,-2.6),(1.0,1.3)\},\\
 A_2&=\{(2.8,-3.2),(4.0,2.6),(0.2,-0.2)\}.
 \end{aligned}
\end{equation}

\begin{figure}[!t]
\centering
\includegraphics[width=\textwidth]{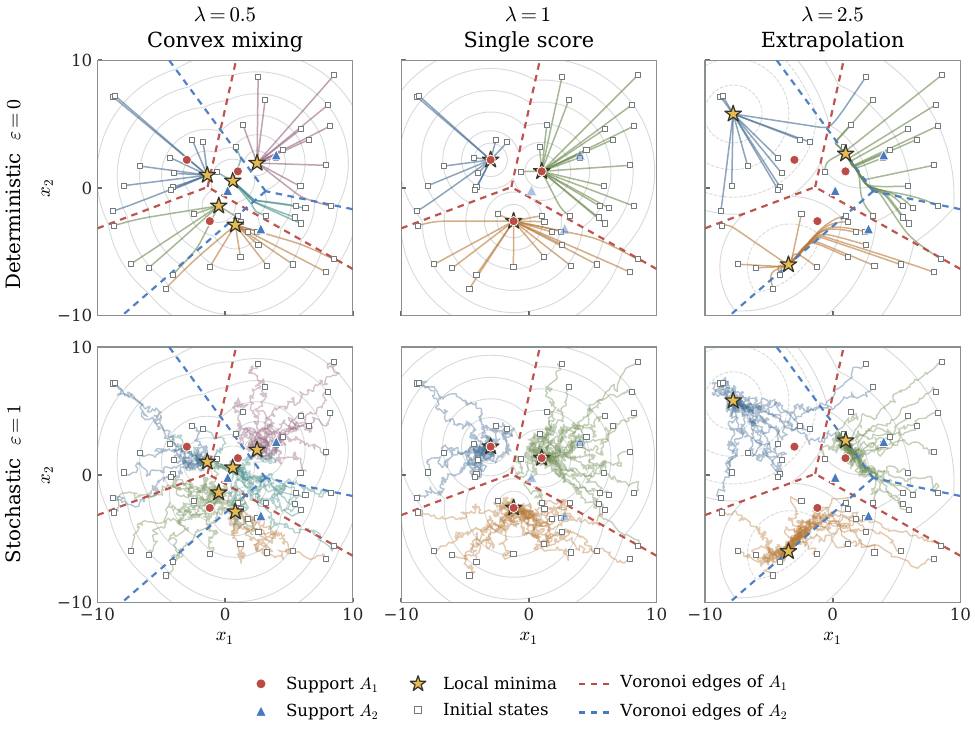}
\caption{Empirical trajectories for \eqref{eq:two-dimensional-supports}.
Columns correspond to $\lambda=0.5,1,2.5$; rows to deterministic
and stochastic trajectories. Stars mark local minima; red and blue
dashed lines mark the Voronoi edges of $A_1$ and $A_2$.
The extrapolative case includes two minima on reference interfaces.
Path colors identify the approached minimum; faint contours
show $\Phi_\lambda$.}
\label{fig:empirical-2d-trajectories}
\end{figure}

Figure~\ref{fig:empirical-2d-trajectories} compares the two dynamics
for $\lambda=0.5,1,2.5$, with $\varepsilon=1$ in the stochastic case.
Both use $48$ common initial states in $[-9,9]^2$.

The marked local-minimum set
$\mathcal M_\lambda\subset\Ecal_\lambda$ contains five weighted
centers for $\lambda=0.5$, equals $A_1$ for $\lambda=1$, and
contains three points for $\lambda=2.5$.
Two of the latter lie on reference Voronoi edges.
The computed paths approach every marked minimum, illustrating
the interface endpoints covered by the convergence theorems.
Stochastic forcing changes the paths and may change the selected
endpoint, while the equilibrium set remains the same.

For deterministic strict-cell endpoints,
\cref{thm:empirical-main,rem:sharp-asymptotics} give
\[
 |Y_\tau-y^*|=O(e^{-\tau/2}),
 \qquad
 |X_t-x^*|=O(\sqrt t).
\]
Figure~\ref{fig:empirical-2d-convergence} compares this scale with
the root-mean-square (RMS) distance to the marked minima:
\begin{equation}\label{eq:empirical-rms-diagnostic}
 R_\lambda^\varepsilon(t)
 =\left(\frac{1}{N_\varepsilon}
   \sum_{r=1}^{N_\varepsilon}
   \dist\bigl(X_t^{\varepsilon,(r)},\mathcal M_\lambda\bigr)^2
   \right)^{1/2},
 \qquad N_0=48,\quad N_1=384.
\end{equation}
The stochastic ensemble uses eight independent noise realizations
per initial state.

Over $10^{-7}\leq t\leq10^{-4}$, the computed RMS curves
exhibit an approximately square-root decay, consistent with
the theoretical rate for deterministic strict-cell endpoints.
For stochastic trajectories and deterministic projection-interface
endpoints, the observed rates remain numerical: the convergence
theorems establish endpoint existence in these cases, while the
corresponding rate questions remain open.

\begin{figure}[!htbp]
\centering
\includegraphics[width=\textwidth]{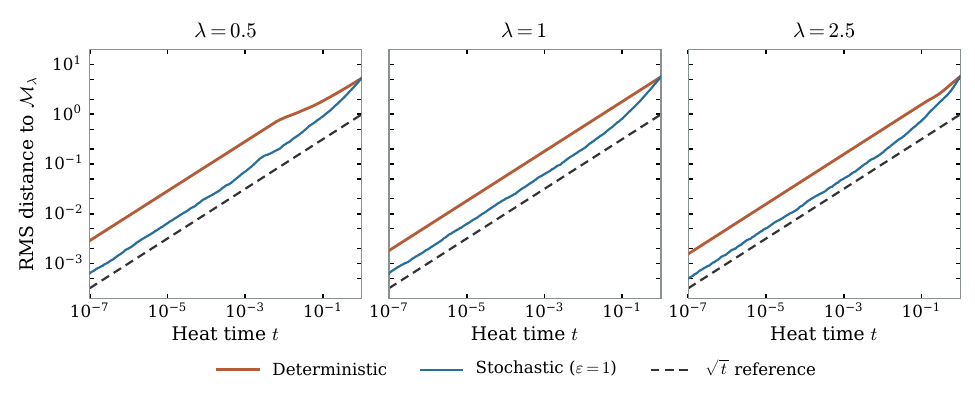}
\caption{Convergence scales for empirical trajectories.
Solid curves show the RMS distances $R_\lambda^0(t)$ and
$R_\lambda^1(t)$ from \eqref{eq:empirical-rms-diagnostic}.
The dashed curve is the reference $\sqrt t$, with unit prefactor.
The data are unsmoothed. The $O(\sqrt t)$ rate is proved for deterministic
strict-cell endpoints; stochastic and projection-interface decay scales
are numerical observations.}
\label{fig:empirical-2d-convergence}
\end{figure}

\FloatBarrier
\section{Proofs of the main results}\label{sec:proofs}

We retain the notation of Sections~\ref{sec:main-results} and
\ref{sec:perturbation-stochastic}. Throughout, $t_\tau=Te^{-\tau}$ and
$\lambda\geq0$ is fixed; $t,t'$ denote heat time and $\tau,\tau',s,s'$
logarithmic time.

\medskip
\noindent\textit{Posterior quantities and normalized potentials.}
For a probability measure $u_0$, write $u=G_t*u_0$ and set
\begin{equation}\label{eq:posterior-general}
 d\nu_t^x(a)=\frac{G_t(x-a)}{u(x,t)}\,du_0(a),\qquad
 m(x,t)=\int a\,d\nu_t^x(a),\qquad f(x,t)=-4t\log u(x,t).
\end{equation}
We also write
\[
 Q(x,t)=\int|x-a|^2\,d\nu_t^x(a),\qquad
 \Sigma(x,t)=\int(a-m(x,t))\otimes(a-m(x,t))\,d\nu_t^x(a).
\]
For $u_{0,k}$, the corresponding quantities carry a subscript $k$.
The posterior-mean identity \cite[Lemma~5.1 and its proof]{liuzuazua2025}
gives
\begin{equation}\label{eq:mean-shift-general}
 \begin{aligned}
 \nabla\log u_k(x,t)&=\frac{m_k(x,t)-x}{2t},\\
 \nabla F_\lambda(x,t)
 &=2\left(x-\lambda m_1(x,t)-(1-\lambda)m_2(x,t)\right).
 \end{aligned}
\end{equation}
Each $m_k(x,t)$ belongs to $\conv(A_k)$, and the deterministic
flow has the form
\begin{equation}\label{eq:posterior-flow}
 \dot Y_\tau=\frac12\bigl(z_\tau-Y_\tau\bigr),\qquad
 z_\tau=\lambda m_1(Y_\tau,t_\tau)
             +(1-\lambda)m_2(Y_\tau,t_\tau)\in K_\lambda.
\end{equation}
For the energy estimates, remove the spatially constant normalization:
\begin{equation}\label{eq:normalized-potential}
 \overline f(x,t)=f(x,t)-2dt\log(4\pi t),\qquad
 \overline F_\lambda(x,t)=F_\lambda(x,t)-2dt\log(4\pi t).
\end{equation}
We write
\[
 V_\tau(x):=\overline F_\lambda(x,t_\tau).
\]
The normalization preserves spatial gradients and vanishes as $t\to0^+$.

\medskip
\noindent\textit{Entropy and information bounds.}
For probability measures, we use the relative entropy
\[
 \KL(\rho\|\mu)=\int\log\!\left(\frac{d\rho}{d\mu}\right)\,d\rho
 \qquad(\rho\ll\mu),
\]
with value $+\infty$ otherwise. For a probability density, set
\[
 \operatorname{Ent}(\rho)=\int_{\R^d}\rho(x)\log\rho(x)\,dx.
\]
For a nonempty compact set $B$, define
\begin{align}
 H_B(t)&:=\sup_{x\in B}\KL(\nu_t^x\|u_0),\qquad
 H_{k,B}(t):=\sup_{x\in B}\KL(\nu_{k,t}^x\|u_{0,k}),
 \label{eq:posterior-information-bounds}\\
 \Ical_\lambda(x,t)&:=\lambda\KL(\nu_{1,t}^x\|u_{0,1})
              +(1-\lambda)\KL(\nu_{2,t}^x\|u_{0,2}).
 \label{eq:signed-information-proof}
\end{align}
As before, terms with zero coefficients are omitted.

\medskip
\noindent\textit{Restarted dynamics.}
For fixed $\varepsilon\geq0$, $Y^{\varepsilon;s,y}$ denotes the solution
of \eqref{eq:similarity-sde} started at time $s$ from $y$, and
$\mathbb E_{s,y}$ its expectation; $\varepsilon=0$ gives the deterministic
flow. We retain $\mathcal G_\tau=\mathcal F_{t_\tau}$ from
\cref{subsec:stochastic-generation} and suppress fixed superscripts
within individual proofs.

\medskip
\noindent\textit{Stochastic variation of constants.}
We use the stochastic variation-of-constants (Duhamel) formula
\cite[Sections~4.4 and~5.1]{DaPratoZabczyk2014}.
For $a>0$, a progressively measurable process $g$ with locally
integrable paths, and a continuous local martingale $M$, consider
\[
 dZ_\tau=(-aZ_\tau+g_\tau)\,d\tau+dM_\tau.
\]
Then, for $0\leq s\leq\tau$, almost surely,
\begin{equation}\label{eq:duhamel}
 Z_\tau=e^{-a(\tau-s)}Z_s
       +\int_s^\tau e^{-a(\tau-s')}g_{s'}\,ds'
       +\int_s^\tau e^{-a(\tau-s')}\,dM_{s'}.
\end{equation}
The case $M=0$ is the deterministic formula.

\subsection{Technical ingredients}\label{subsec:technical-statements}

The five lemmas below are used in the proofs of the main results.
Their proofs are given in Section~\ref{sec:technical-proofs}.

\begin{lemma}[Distance geometry and empirical equilibria]
\label[lemma]{lem:distance-geometry}
Assume that the active supports are compact.
\begin{enumerate}[label=\textup{(\roman*)}]
\item For every nonempty compact $A$, the function $d_A^2$ is locally
Lipschitz and $2$-semiconcave, with
\begin{equation}\label{eq:distance-clarke-formula}
 \partial^C d_A^2(x)=2\left(x-\conv\Pi_A(x)\right).
\end{equation}
The separated envelope is upper semicontinuous with nonempty compact
convex values, and
\[
 \begin{cases}
 \widehat\partial\Phi_\lambda=\partial^C\Phi_\lambda,
       &0\leq\lambda\leq1,\\
 \partial^C\Phi_\lambda\subseteq\widehat\partial\Phi_\lambda,
       &\lambda>1.
 \end{cases}
\]
For $0\leq\lambda\leq1$, the potential $\Phi_\lambda$ is $2$-semiconcave.
\item For finite empirical active laws, each collection of nearest-index
sets determines at most one equilibrium. In particular,
\[
 \#\Ecal_\lambda<\infty.
\]
\end{enumerate}
\end{lemma}

\begin{lemma}[Heat-flow Hessians and the Li--Yau bound]
\label[lemma]{lem:heat-hessian}
Assume that the active supports are compact. For $t>0$,
\begin{equation}\label{eq:posterior-covariance}
 \begin{aligned}
 D_x^2\log u_k(x,t)&=-\frac{I_d}{2t}+\frac{\Sigma_k(x,t)}{4t^2},\\
 D_x^2F_\lambda(x,t)&=2I_d-\frac1t
       \left(\lambda\Sigma_1(x,t)+(1-\lambda)\Sigma_2(x,t)\right).
 \end{aligned}
\end{equation}
In particular,
\[
 D_x^2\log u_k\succeq-\frac{I_d}{2t},\qquad
 D_x^2F_\lambda\preceq2I_d\quad\text{if }0\leq\lambda\leq1.
\]
On every strip $\delta\leq t\leq T$ with $\delta>0$, the mixed score
and $\nabla F_\lambda$ are globally Lipschitz in space and have at
most linear growth, uniformly in $t$.
\end{lemma}

\begin{lemma}[Laplace--Varadhan localization]
\label[lemma]{lem:compact-soft-min}
\leavevmode\par\nobreak
\begin{enumerate}[label=\textup{(\roman*)}]
\item Let $u_0$ have nonempty compact support $A$. Then
\[
 f(\cdot,t)\longrightarrow d_A^2
 \qquad\text{locally uniformly as }t\to0^+.
\]
For any sequences $x_n\to x$ and $t_n\to0^+$,
\[
 m(x_n,t_n)\longrightarrow a
 \quad\Longrightarrow\quad a\in\conv\Pi_A(x).
\]
\item For finite empirical active laws and $x_0\notin\Scal_\lambda$,
let $a_{k,*}$ be the active nearest atoms and set
\[
 c_*=\lambda a_{1,*}+(1-\lambda)a_{2,*}.
\]
There exist a neighborhood $U$ of $x_0$ and $C,\eta,t_0>0$ such that
\begin{equation}\label{eq:local-exponential-score}
 \sup_{x\in U}|\nabla F_\lambda(x,t)-2(x-c_*)|
       \leq Ce^{-\eta/t},\qquad 0<t\leq t_0.
\end{equation}
\end{enumerate}
\end{lemma}

\Needspace{6\baselineskip}
\begin{lemma}[Passage to the geometric inclusion]
\label[lemma]{lem:translated-profiles}
Assume compact active supports. Let $\kappa>0$ and let continuous
curves $y,R:[0,\infty)\to\R^d$, with $y$ bounded and $R_0=0$, satisfy
\begin{equation}\label{eq:profile-integral-equation}
 y_\tau=y_0-\frac\kappa4\int_0^\tau
       \nabla F_\lambda(y_s,t_s)\,ds+R_\tau.
\end{equation}
If, for every $S>0$,
\begin{equation}\label{eq:profile-forcing-increments}
 \sup_{0\leq h\leq S}|R_{\tau+h}-R_\tau|\longrightarrow0
 \qquad(\tau\to\infty),
\end{equation}
then, for any $\tau_n\to\infty$, the translates $y_{\tau_n+\cdot}$
are relatively compact in $C_{\mathrm{loc}}([0,\infty);\R^d)$.
Every limit $z$ is locally Lipschitz and satisfies
\[
 \dot z_s\in-\frac\kappa4\widehat\partial\Phi_\lambda(z_s)
 \qquad\text{for a.e. }s\geq0.
\]
\end{lemma}

The final lemma controls the time dependence of the potential and the
conditional energy increments of stochastic trajectories.

\Needspace{7\baselineskip}
\begin{lemma}[Deterministic and stochastic energy estimates]
\label[lemma]{lem:energy-estimates}
Assume the hypotheses on the data and $\lambda$ in
\cref{thm:compact-omega}.
\begin{enumerate}[label=\textup{(\roman*)}]
\item For every nonempty compact $B\subset\R^d$, there exists a
nonnegative function $b_B\in L^1(0,\infty)$ such that
\begin{equation}\label{eq:potential-time-control}
 \partial_\tau V_\tau(x)=-4t_\tau\Ical_\lambda(x,t_\tau)
       \leq b_B(\tau),\qquad x\in B,\quad\tau\geq0.
\end{equation}
For $0\leq\lambda\leq1$, one may take $b_B=0$.
\item For every fixed $\varepsilon>0$, there exist $s_0,C>0$ such
that, for all $s\geq s_0$ and $y\in\R^d$,
\begin{equation}\label{eq:conditional-energy-step}
 \mathbb E_{s,y}\!\bigl[V_{s+1}(Y_{s+1}^{\varepsilon;s,y})\bigr]
 \leq V_s(y)+Ct_s(1+|\log t_s|)(1+|y|^2).
\end{equation}
The constants may depend on $\varepsilon,\lambda,T$ and the data laws,
but are uniform in $s$ and $y$.
\end{enumerate}
\end{lemma}

\subsection{Proof of Theorem~\ref{thm:support-reduction}}
\label{subsec:geometry-proof}
\leavevmode\par\nobreak

\proofstep{Step 1}{Localization and graphical recovery}
The identities
\[
 F_\lambda=\lambda f_1+(1-\lambda)f_2,
 \qquad
 \Phi_\lambda=\lambda d_1^2+(1-\lambda)d_2^2,
\]
and \cref{lem:compact-soft-min} give local uniform convergence of the
potentials. To identify the field limits, let $x_n\to x$, $t_n\to0^+$,
and suppose that $\nabla F_\lambda(x_n,t_n)\to p$. Each active posterior
mean lies in a fixed compact support hull. Along a common subsequence,
\[
 m_k(x_n,t_n)\longrightarrow a_k\in\conv\Pi_k(x).
\]
The mean-shift identity therefore yields
\[
 p=2\left(x-\lambda a_1-(1-\lambda)a_2\right)
   \in\widehat\partial\Phi_\lambda(x).
\]
For the uniform graph-distance assertion on a compact set $B$, a
contrary sequence would satisfy
\[
 x_n\in B,\qquad t_n\to0^+,\qquad
 \dist\!\left((x_n,\nabla F_\lambda(x_n,t_n)),
                   \gph\widehat\partial\Phi_\lambda\right)\geq\delta
\]
for some $\delta>0$. Compactness and the bounded posterior means give a
subsequence converging to $(x,p)$ in the envelope graph, contradicting
this lower bound. This proves \eqref{eq:compact-graph-convergence}.

For graphical recovery, assume $0\leq\lambda\leq1$, fix
$x\in\R^d$ and $p\in\partial^C\Phi_\lambda(x)$, and let
$t_n\to0^+$ be any sequence. Choose $B=\overline B(x,r)$ with $r>0$
and define on this ball
\[
 \begin{aligned}
 J_n(y)&=F_\lambda(y,t_n)-\langle p,y-x\rangle-2|y-x|^2,\\
 J(y)&=\Phi_\lambda(y)-\langle p,y-x\rangle-2|y-x|^2.
 \end{aligned}
\]
The curvature and semiconcavity bounds in
\cref{lem:heat-hessian,lem:distance-geometry} give
\[
 D^2J_n\preceq-2I_d,\qquad J(y)\leq J(x)-|y-x|^2.
\]
Let $x_n'$ be the unique maximizer of $J_n$ on $B$ and set
\[
 \delta_n:=\sup_{y\in B}|F_\lambda(y,t_n)-\Phi_\lambda(y)|.
\]
Then
\[
 J(x)-\delta_n\leq J_n(x)\leq J_n(x_n')
       \leq J(x)-|x_n'-x|^2+\delta_n.
\]
Hence $|x_n'-x|^2\leq2\delta_n\to0$. For large $n$, the maximizer lies
in the interior of $B$, and its first-order condition gives
\[
 \nabla F_\lambda(x_n',t_n)-p-4(x_n'-x)=0,
 \qquad \nabla F_\lambda(x_n',t_n)\longrightarrow p.
\]
This proves \eqref{eq:convex-graph-recovery}. On a region with unique
active projections, the envelope is a continuous singleton. The same
subsequence argument on each compact subset gives uniform convergence
of the fields there.

\proofstep{Step 2}{Confinement and limiting trajectories}
On each bounded logarithmic-time interval, \cref{lem:heat-hessian}
gives a unique solution. Set
\[
 z_s=\lambda m_1(Y_s,t_s)+(1-\lambda)m_2(Y_s,t_s)\in K_\lambda.
\]
Applying \eqref{eq:duhamel} with $a=1/2$ and $M=0$ gives
\begin{equation}\label{eq:variation-compact}
 Y_\tau=e^{-\tau/2}y_0+\frac12\int_0^\tau e^{-(\tau-s)/2}z_s\,ds
 \in e^{-\tau/2}y_0+(1-e^{-\tau/2})K_\lambda.
\end{equation}
The normalized kernel average $k_\tau$ belongs to $K_\lambda$.
Hence $Y$ stays in $\conv(\{y_0\}\cup K_\lambda)$, extends globally,
and has bounded derivative.

Let $p_0$ be the projection of $y_0$ onto $K_\lambda$. Since
$e^{-\tau/2}p_0+(1-e^{-\tau/2})k_\tau\in K_\lambda$, we obtain
\begin{equation}\label{eq:distance-to-K}
 \dist(Y_\tau,K_\lambda)
 \leq e^{-\tau/2}|y_0-p_0|
 =e^{-\tau/2}\dist(y_0,K_\lambda).
\end{equation}
Apply \cref{lem:translated-profiles} with $\kappa=1$ and $R=0$.
Every translated limit solves \eqref{eq:limit-inclusion-main}; the last
estimate places its whole range in $K_\lambda$.

\proofstep{Step 3}{Signed dissipation}
Let $Z$ solve \eqref{eq:limit-inclusion-main}. On each compact time
interval, local boundedness of the envelope makes $Z$ and every active
composition $d_k^2\circ Z$ Lipschitz. Fix a time $s$ at which these
curves are differentiable and the inclusion holds; such times form
a set of full measure.

For any $p_k\in\partial^Cd_k^2(Z_s)$,
\eqref{eq:distance-supergradient-inequality} gives
\[
 d_k^2(Z_{s+h})-d_k^2(Z_s)
 \leq\langle p_k,Z_{s+h}-Z_s\rangle+|Z_{s+h}-Z_s|^2.
\]
At the chosen time,
\[
 Z_{s+h}-Z_s=h\dot Z_s+o(|h|),\qquad
 \frac{|Z_{s+h}-Z_s|^2}{|h|}\longrightarrow0.
\]
Division by $h>0$ gives
$(d_k^2\circ Z)'(s)\leq\langle p_k,\dot Z_s\rangle$; division by
$h<0$ gives the reverse inequality. Consequently,
\[
 (d_k^2\circ Z)'(s)=\langle p_k,\dot Z_s\rangle
 \qquad\text{for every }p_k\in\partial^Cd_k^2(Z_s).
\]
At each such time, the inclusion supplies component vectors satisfying
\[
 \lambda p_1+(1-\lambda)p_2=-4\dot Z_s.
\]
Multiplying the component identities by their signed coefficients gives
\begin{equation}\label{eq:signed-energy-identity}
 \frac{d}{ds}\Phi_\lambda(Z_s)
 =\bigl\langle\lambda p_1+(1-\lambda)p_2,\dot Z_s\bigr\rangle
 =-4|\dot Z_s|^2.
\end{equation}
The composition $\Phi_\lambda\circ Z$ is absolutely continuous, so
integration proves \eqref{eq:profile-energy}. The componentwise equality
is what preserves dissipation when $1-\lambda$ is negative.

\subsection{Proof of Theorem~\ref{thm:compact-omega}}
\label{subsec:compact-proof}
\leavevmode\par\nobreak

\proofstep{Step 1}{The characteristic energy and reference information}
The trajectory stays in $B=\conv(\{y_0\}\cup K_\lambda)$ by
\eqref{eq:variation-compact}. The chain rule and
\cref{lem:energy-estimates}\textup{(i)} give
\begin{equation}\label{eq:characteristic-energy-clean}
 \frac{d}{d\tau}V_\tau(Y_\tau)
 =-4|\dot Y_\tau|^2+R_\lambda(\tau),\qquad
 R_\lambda(\tau):=-4t_\tau\Ical_\lambda(Y_\tau,t_\tau).
\end{equation}
For convex mixing, nonnegativity of relative entropy gives
$R_\lambda\leq0$. For $\lambda>1$, the target contribution remains
nonpositive, and the reference contribution satisfies
\[
 R_\lambda(\tau)
 \leq4(\lambda-1)t_\tau
          \KL(\nu_{2,t_\tau}^{Y_\tau}\|u_{0,2})
 \leq4(\lambda-1)t_\tau H_{2,B}(t_\tau).
\]
Accordingly, set
\[
 b(\tau)=
 \begin{cases}
  0,&0\leq\lambda\leq1,\\
  4(\lambda-1)t_\tau H_{2,B}(t_\tau),&\lambda>1.
 \end{cases}
\]
Thus $R_\lambda\leq b$ and $b\geq0$. For $\lambda>1$,
the bound \eqref{eq:compact-KL-estimate} and $d\tau=-dt/t$ give
\[
 \int_0^\infty b(\tau)\,d\tau
 =4(\lambda-1)\int_0^T H_{2,B}(t)\,dt<\infty.
\]
Indeed, $H_{2,B}(t)\leq C_B(1+|\log t|)$ near zero and is bounded
on compact positive-time intervals.

\proofstep{Step 2}{Corrected energy and finite action}
Define
\begin{equation}\label{eq:corrected-energy-general}
 \mathcal L(\tau)=\overline F_\lambda(Y_\tau,t_\tau)
                        +\int_\tau^\infty b(s)\,ds.
\end{equation}
This function is locally absolutely continuous and satisfies
\begin{equation}\label{eq:corrected-energy-dissipation}
 \mathcal L'(\tau)=-4|\dot Y_\tau|^2+R_\lambda(\tau)-b(\tau)
                  \leq-4|\dot Y_\tau|^2
\end{equation}
for almost every $\tau$. Local uniform convergence to $\Phi_\lambda$
and continuity on positive-time strips bound the potential below on
$B$. Hence $\mathcal L$ is bounded below and decreases to a finite
limit $\ell$.

Integrating the differential inequality over $[\tau,S]$ yields
\[
 4\int_\tau^S|\dot Y_s|^2\,ds
 \leq\mathcal L(\tau)-\mathcal L(S).
\]
Letting $S\to\infty$ proves
\begin{equation}\label{eq:tail-kinetic-action}
 4\int_\tau^\infty|\dot Y_s|^2\,ds\leq\mathcal L(\tau)-\ell.
\end{equation}
In particular, the action is finite and its tail tends to zero.
Moreover,
\[
 F_\lambda(Y_\tau,t_\tau)
 =\mathcal L(\tau)-\int_\tau^\infty b(s)\,ds
                +2dt_\tau\log(4\pi t_\tau)\longrightarrow\ell.
\]
Uniform convergence of $F_\lambda(\cdot,t)$ on $B$ then gives
$\Phi_\lambda(Y_\tau)\to\ell$, proving \eqref{eq:common-energy-limit}.

\proofstep{Step 3}{Stationary limits and equilibrium localization}
For each $S>0$, absolute continuity and Cauchy--Schwarz give
\begin{equation}\label{eq:stationary-translates}
 \begin{aligned}
 \sup_{0\leq h\leq S}|Y_{\tau+h}-Y_\tau|
 &\leq\int_\tau^{\tau+S}|\dot Y_s|\,ds\\
 &\leq\sqrt S\left(\int_\tau^\infty|\dot Y_s|^2\,ds\right)^{1/2}
 \longrightarrow0.
 \end{aligned}
\end{equation}
Take $z\in\omega(Y)$ and choose $\tau_n\to\infty$ with $Y_{\tau_n}\to z$.
The estimate implies
\[
 Y_{\tau_n+\cdot}\longrightarrow(s\mapsto z)
 \quad\text{locally uniformly on }[0,\infty).
\]
By \cref{thm:support-reduction}, this constant curve solves the limiting
inclusion. Therefore
\[
 0\in\widehat\partial\Phi_\lambda(z),\qquad
 z\in\Ecal_\lambda,\qquad \Phi_\lambda(z)=\ell,
\]
where the last identity follows from continuity and the energy limit.

Finally, the closed tails $B_S=\overline{Y([S,\infty))}$ are nested,
nonempty compact connected sets. Their intersection $\omega(Y)$ has
these same properties, proving \eqref{eq:compact-omega-localization}.

\subsection{Proof of Theorem~\ref{thm:empirical-main}}
\label{subsec:empirical-proof}
\leavevmode\par\nobreak

\proofstep{Step 1}{Finite equilibria give point convergence}
Positive atomic weights imply the small-ball condition with exponent
zero, so the compact-support theorem applies for every $\lambda\geq0$.
Its localization conclusion and \cref{lem:distance-geometry}\textup{(ii)} give
\[
 \varnothing\neq\omega(Y)\subseteq\Ecal_\lambda,\qquad
 \omega(Y)\ \text{connected},\qquad \#\Ecal_\lambda<\infty.
\]
Hence $\omega(Y)=\{y^*\}$, and precompactness implies $Y_\tau\to y^*$
from every initial state. The inverse time change gives
\eqref{eq:empirical-physical-convergence} with $x^*:=y^*$.

\proofstep{Step 2}{The affine tail and its asymptotic coefficient}
If $y^*\notin\Scal_\lambda$, convergence places the tail in the
neighborhood of \cref{lem:compact-soft-min}\textup{(ii)}. The equilibrium
relation gives $c_*=y^*$, so, for $e_\tau=Y_\tau-y^*$ and $\tau\geq\tau_0$,
\[
 \dot e_\tau=-\frac12e_\tau+r_\tau,\qquad
 |r_\tau|\leq Ce^{-\eta/t_\tau}=Ce^{-\eta e^\tau/T}.
\]
Formula~\eqref{eq:duhamel}, applied from time $\tau_0$, gives
\begin{equation}\label{eq:strict-variation}
 e_\tau=e^{-\tau/2}
 \left(e^{\tau_0/2}e_{\tau_0}
             +\int_{\tau_0}^\tau e^{s/2}r_s\,ds\right).
\end{equation}
With $t'=Te^{-s}$, the weighted remainder satisfies
\[
 \int_{\tau_0}^\infty e^{s/2}|r_s|\,ds
 \leq C\sqrt T\int_0^{t_{\tau_0}}(t')^{-3/2}e^{-\eta/t'}\,dt'
 <\infty.
\]
The limiting coefficient is
\[
 \xi^*=e^{\tau_0/2}e_{\tau_0}
                   +\int_{\tau_0}^\infty e^{s/2}r_s\,ds,
 \qquad e^{\tau/2}e_\tau\longrightarrow\xi^*.
\]
Subtracting \eqref{eq:strict-variation} from this definition gives the
exact remainder formula
\[
 e_\tau-e^{-\tau/2}\xi^*
 =-e^{-\tau/2}\int_\tau^\infty e^{s/2}r_s\,ds.
\]
For $t=t_\tau$, use the same substitution and
$(t')^{-3/2}\leq\sqrt t\,(t')^{-2}$ on $(0,t]$:
\[
 \begin{aligned}
 \left|e_\tau-e^{-\tau/2}\xi^*\right|
 &\leq C\sqrt t\int_0^t(t')^{-3/2}e^{-\eta/t'}\,dt'\\
 &\leq Ct\int_0^t(t')^{-2}e^{-\eta/t'}\,dt'
 =\frac C\eta\,t e^{-\eta/t}.
 \end{aligned}
\]
This proves \eqref{eq:tau-sharp-rate}. The expansion bounds
$e^{\tau/2}|Y_\tau-y^*|$ for large $\tau$; continuity on the remaining
compact time interval gives \eqref{eq:strict-cell-rate} for all
$\tau\geq0$ after enlarging $C$.
Since $e^{-\tau/2}=\sqrt{t_\tau/T}$, the same expansion gives
\eqref{eq:t-sharp-rate}. The local balance gives
$y^*=\lambda a_{1,*}+(1-\lambda)a_{2,*}$. If $\xi^*\neq0$, then
\[
 e^{\tau/2}|Y_\tau-y^*|\longrightarrow|\xi^*|>0,
 \qquad
 \frac{|X_t-x^*|}{\sqrt t}\longrightarrow\frac{|\xi^*|}{\sqrt T}.
\]
Thus the leading scale is attained, with constants depending on the
trajectory's entrance into the strict cell.

\subsection{Proof of Proposition~\ref{prop:L2-robustness}}
\label{subsec:robustness-proof}
\leavevmode\par\nobreak

\proofstep{Step 1}{Confinement under square-integrable forcing}
Positive-time Lipschitz bounds and $r\in L^1_{\mathrm{loc}}$ give a unique
maximal Carath\'eodory solution. With
$z_s=\lambda m_1(\widetilde Y_s,t_s)+(1-\lambda)m_2(\widetilde Y_s,t_s)
\in K_\lambda$, \eqref{eq:duhamel} gives
\begin{equation}\label{eq:perturbed-variation}
 \begin{aligned}
 \widetilde Y_\tau
 &=e^{-\tau/2}\widetilde y_0
   +\frac12\int_0^\tau e^{-(\tau-s)/2}z_s\,ds+\mathcal R(\tau),\\
 \mathcal R(\tau)&=\int_0^\tau e^{-(\tau-s)/2}r_s\,ds.
 \end{aligned}
\end{equation}
The first two terms lie in the convex hull of
$\{\widetilde y_0\}\cup K_\lambda$. The forcing contribution satisfies
\[
 |\mathcal R(\tau)|
 \leq\left(\int_0^\tau e^{-(\tau-s)}\,ds\right)^{1/2}
                  \|r\|_{L^2(0,\tau)}
 \leq\|r\|_{L^2(0,\infty)}.
\]
Thus the solution is global and confined to a compact ball $B$.
For $\tau\geq S$, split the convolution at $S$:
\[
 |\mathcal R(\tau)|\leq e^{-(\tau-S)/2}|\mathcal R(S)|
                          +\|r\|_{L^2(S,\infty)}.
\]
Letting $\tau\to\infty$ and then $S\to\infty$ gives
$\mathcal R(\tau)\to0$. The convex-hull estimate yields
\begin{equation}\label{eq:perturbed-confinement-bound}
 \dist(\widetilde Y_\tau,K_\lambda)
 \leq e^{-\tau/2}\dist(\widetilde y_0,K_\lambda)+|\mathcal R(\tau)|
 \longrightarrow0.
\end{equation}

\proofstep{Step 2}{Absorption of the energy error}
The bounded drift and $r\in L^2$ give
$\dot{\widetilde Y}\in L^2_{\mathrm{loc}}$, so the chain rule applies.
The identity \eqref{eq:potential-time-control} applies along
the perturbed curve. Let $b\geq0$ be the integrable bound from
\cref{subsec:compact-proof} on $B$, so that
\[
 \widetilde R_\lambda(\tau)
 =-4t_\tau\Ical_\lambda(\widetilde Y_\tau,t_\tau)\leq b(\tau).
\]
The identity $\nabla F_\lambda=-4(\dot{\widetilde Y}-r)$ and the estimate
\[
 4\langle r_\tau,\dot{\widetilde Y}_\tau\rangle
 \leq2|r_\tau|^2+2|\dot{\widetilde Y}_\tau|^2
\]
give
\begin{equation}\label{eq:perturbed-energy-bound}
 \begin{aligned}
 \frac{d}{d\tau}\overline F_\lambda(\widetilde Y_\tau,t_\tau)
 &=-4|\dot{\widetilde Y}_\tau|^2
    +4\langle r_\tau,\dot{\widetilde Y}_\tau\rangle
    +\widetilde R_\lambda(\tau)\\
 &\leq-2|\dot{\widetilde Y}_\tau|^2+2|r_\tau|^2+b(\tau).
 \end{aligned}
\end{equation}
Thus the corrected energy
\begin{equation}\label{eq:perturbed-corrected-energy}
 \widetilde{\mathcal L}(\tau)
 =\overline F_\lambda(\widetilde Y_\tau,t_\tau)
       +\int_\tau^\infty(2|r_s|^2+b(s))\,ds
\end{equation}
is locally absolutely continuous, bounded below, and satisfies
\[
 \widetilde{\mathcal L}'(\tau)
 \leq-2|\dot{\widetilde Y}_\tau|^2\qquad\text{for a.e. }\tau.
\]
It decreases to a finite limit $\widetilde\ell$. Integration first over
a finite interval and then to infinity yields
\begin{equation}\label{eq:perturbed-tail-action}
 2\int_\tau^\infty|\dot{\widetilde Y}_s|^2\,ds
 \leq\widetilde{\mathcal L}(\tau)-\widetilde\ell.
\end{equation}
The error tail and Gaussian normalization vanish as $\tau\to\infty$.
Local uniform convergence of the potentials on $B$ then gives both
energy limits in \eqref{eq:perturbed-energy-limits}.

\proofstep{Step 3}{Equilibrium localization}
The tail-action estimate implies, for each $S>0$,
\[
 \sup_{0\leq h\leq S}
       |\widetilde Y_{\tau+h}-\widetilde Y_\tau|
 \leq\sqrt S\left(\int_\tau^\infty
             |\dot{\widetilde Y}_s|^2\,ds\right)^{1/2}
 \longrightarrow0.
\]
To identify the constant limiting curves, apply
\cref{lem:translated-profiles} with $\kappa=1$ and
$R_\tau=\int_0^\tau r_s\,ds$. Its forcing condition follows from
\[
 \sup_{0\leq h\leq S}|R_{\tau+h}-R_\tau|
 \leq\sqrt S\,\|r\|_{L^2(\tau,\infty)}\longrightarrow0.
\]
If $\widetilde Y_{\tau_n}\to z$, the constant translated limit solves
the inclusion, and energy convergence gives
\[
 0\in\widehat\partial\Phi_\lambda(z),\qquad
 z\in\Ecal_\lambda,\qquad \Phi_\lambda(z)=\widetilde\ell.
\]
The closed-tail argument proves \eqref{eq:perturbed-omega}.
For empirical data, \cref{lem:distance-geometry}\textup{(ii)} makes this connected
set a singleton, proving point convergence in
\cref{rem:physical-score-robustness}.

\subsection{Proofs of Theorem~\ref{thm:stochastic-convergence} and Corollary~\ref{cor:terminal-law}}
\label{subsec:stochastic-proof}

We first prove \cref{thm:stochastic-convergence}.
For $\varepsilon=0$, the assertions are
\cref{thm:compact-omega,thm:empirical-main}. Assume $\varepsilon>0$ and
write
\[
 Y=Y^\varepsilon,\qquad \kappa=1+\varepsilon,\qquad a=\kappa/2.
\]
We use \cref{lem:energy-estimates}\textup{(ii)} to prove energy
convergence, then apply the pathwise compactness and dissipation
arguments.

\proofstep{Step 1}{Bounded paths and second moments}
The accumulated noise is
\begin{equation}\label{eq:stochastic-martingale}
 M_\tau=\int_0^\tau\sqrt{2\varepsilon t_s}\,dW_s,
 \qquad \mathbb E|M_\tau|^2=2\varepsilon dT(1-e^{-\tau}).
\end{equation}
By the It\^o isometry, $M$ is a continuous $L^2$-bounded martingale
and converges almost surely \cite{revuzyor1999continuous}.
For each continuous noise path, $H=Y-M$ solves
\[
 \dot H_\tau=-\frac\kappa4\nabla F_\lambda(H_\tau+M_\tau,t_\tau),
 \qquad H_0=y_0,\qquad Y=H+M.
\]
On finite intervals, \cref{lem:heat-hessian} gives a continuous,
globally Lipschitz drift with linear growth. The solutions agree on
overlaps and depend only on the past noise, yielding the unique global
strong solution \cite{karatzas1991brownian}.

With $z_s=\lambda m_1(Y_s,t_s)+(1-\lambda)m_2(Y_s,t_s)\in K_\lambda$,
\eqref{eq:duhamel} and integration by parts give
\begin{equation}\label{eq:stochastic-variation-compact}
 \begin{aligned}
 Y_\tau&=e^{-a\tau}y_0+a\int_0^\tau e^{-a(\tau-s)}z_s\,ds+C_\tau,\\
 C_\tau&=\int_0^\tau e^{-a(\tau-s)}\,dM_s
       =M_\tau-a\int_0^\tau e^{-a(\tau-s)}M_s\,ds.
 \end{aligned}
\end{equation}
Whenever $M_\tau\to M^*$, its exponential average has the same limit.
Indeed, for $\tau\geq S$,
\[
 \left|a\int_0^\tau e^{-a(\tau-s)}(M_s-M^*)\,ds\right|
 \leq e^{-a(\tau-S)}\sup_{0\leq s\leq S}|M_s-M^*|
       +\sup_{s\geq S}|M_s-M^*|.
\]
First let $\tau\to\infty$, then $S\to\infty$. It follows that
$C_\tau\to0$. Since the first two terms in
\eqref{eq:stochastic-variation-compact} belong to
$e^{-a\tau}y_0+(1-e^{-a\tau})K_\lambda$, we obtain
\begin{equation}\label{eq:stochastic-confinement}
 \dist(Y_\tau,K_\lambda)
 \leq e^{-a\tau}\dist(y_0,K_\lambda)+|C_\tau|\longrightarrow0.
\end{equation}
The same formula bounds the path almost surely. Moreover,
\[
 \mathbb E|C_\tau|^2
 =2\varepsilon d\int_0^\tau e^{-2a(\tau-s)}t_s\,ds
 \leq\frac{\varepsilon dT}{a},
\]
so compactness of $K_\lambda$ also gives
\begin{equation}\label{eq:stochastic-uniform-moment}
 \sup_{\tau\geq0}\mathbb E|Y_\tau|^2<\infty.
\end{equation}

\proofstep{Step 2}{Convergence of the energy at integer times}
The variable $Y_n$ is $\mathcal G_n$-measurable.
Compact support and Jensen's inequality give
\begin{equation}\label{eq:stochastic-potential-growth}
 \tfrac12|x|^2-C_0\leq V_\tau(x),\qquad
 |V_\tau(x)|\leq C_0(1+|x|^2),
\end{equation}
uniformly in $\tau$; see \cref{subsec:energy-technical}.
Thus $U_n:=V_n(Y_n)+C_0$ is nonnegative and integrable by
\eqref{eq:stochastic-uniform-moment}.

Conditionally on $\mathcal G_n$, the process restarts from $Y_n$ at
time $n$. The Markov property and \cref{lem:energy-estimates}\textup{(ii)} give,
for all $n\geq n_0$ with $n_0$ sufficiently large,
\begin{equation}\label{eq:sampled-almost-supermartingale}
 \mathbb E[U_{n+1}\mid\mathcal G_n]\leq U_n+B_n,
 \qquad B_n:=Ct_n(1+|\log t_n|)(1+|Y_n|^2).
\end{equation}
The errors $B_n$ are nonnegative and $\mathcal G_n$-measurable.
The uniform moment bound and $t_n=Te^{-n}$ imply
\[
 \sum_{n\geq n_0}\mathbb E B_n
 \leq C\left(1+\sup_{\tau\geq0}\mathbb E|Y_\tau|^2\right)
        \sum_{n\geq n_0}t_n(1+|\log t_n|)<\infty.
\]
By Tonelli's theorem, the nonnegative sum $\sum_{n\geq n_0}B_n$ belongs
to $L^1$. We can therefore add the remaining error to the energy through
its conditional expectation:
\begin{equation}\label{eq:stochastic-corrected-energy}
 A_n:=\mathbb E\!\left[\sum_{k=n}^{\infty}B_k
                 \,\middle|\,\mathcal G_n\right],
 \qquad L_n:=U_n+A_n.
\end{equation}
These variables are nonnegative and integrable. The tower property,
$\mathcal G_n\subseteq\mathcal G_{n+1}$, and measurability of $B_n$ give
\[
 \begin{aligned}
 \mathbb E[A_{n+1}\mid\mathcal G_n]
   &=\mathbb E\!\left[\sum_{k=n+1}^{\infty}B_k
                    \,\middle|\,\mathcal G_n\right]=A_n-B_n,\\
 \mathbb E[L_{n+1}\mid\mathcal G_n]
   &\leq U_n+B_n+A_n-B_n=L_n.
 \end{aligned}
\]
Thus $A_n$ and $L_n$ are nonnegative supermartingales and have finite
almost-sure limits \cite{revuzyor1999continuous,robbins1971convergence}.
Moreover,
\[
 \mathbb E A_n=\sum_{k=n}^{\infty}\mathbb E B_k\longrightarrow0.
\]
Writing $A_\infty=\lim_{n\to\infty}A_n\geq0$, Fatou's lemma gives
\[
 0\leq\mathbb E A_\infty
 \leq\liminf_{n\to\infty}\mathbb E A_n=0.
\]
Hence $A_n\to0$ almost surely. Consequently,
\[
 V_n(Y_n)=L_n-A_n-C_0\longrightarrow\ell^\varepsilon
 \qquad\text{almost surely}
\]
for a finite random variable $\ell^\varepsilon$.
Finally, local uniform convergence $V_n\to\Phi_\lambda$ and pathwise
boundedness imply
\[
 |V_n(Y_n)-\Phi_\lambda(Y_n)|\longrightarrow0
 \qquad\text{almost surely}.
\]
Therefore
\begin{equation}\label{eq:stochastic-integer-energy}
 \Phi_\lambda(Y_n)\longrightarrow\ell^\varepsilon
 \qquad\text{almost surely}.
\end{equation}

\proofstep{Step 3}{All times and equilibrium localization}
Fix a continuous path for which the bounds and limits above hold;
this is the case almost surely. The remaining argument is deterministic.
The integrated equation is
\begin{equation}\label{eq:stochastic-integrated-path}
 Y_\tau=y_0-\frac\kappa4\int_0^\tau
                 \nabla F_\lambda(Y_s,t_s)\,ds+M_\tau.
\end{equation}
Convergence of $M$ implies, for every $S>0$,
\begin{equation}\label{eq:stochastic-vanishing-increments}
 \sup_{0\leq h\leq S}|M_{\tau+h}-M_\tau|\longrightarrow0.
\end{equation}
Lemma~\ref{lem:translated-profiles} therefore gives compactness of the
translated paths and identifies each limit $z$ as a solution of
\[
 \dot z_s\in-\frac\kappa4\widehat\partial\Phi_\lambda(z_s).
\]
The signed chain rule in \cref{subsec:geometry-proof} yields
\begin{equation}\label{eq:stochastic-profile-energy}
 \Phi_\lambda(z_{\tau'})+\frac4\kappa
       \int_\tau^{\tau'}|\dot z_s|^2\,ds
 =\Phi_\lambda(z_\tau),\qquad 0\leq\tau\leq\tau'.
\end{equation}

For a limit of integer translates $Y_{n_j+\cdot}$ on $[0,1]$,
\eqref{eq:stochastic-integer-energy} gives
\[
 \Phi_\lambda(z_0)=\Phi_\lambda(z_1)=\ell^\varepsilon.
\]
The dissipation identity forces $z$ to be constant. It follows that
\begin{equation}\label{eq:stochastic-unit-oscillation}
 \sup_{0\leq h\leq1}|Y_{n+h}-Y_n|\longrightarrow0.
\end{equation}
A violating sequence would have a nonconstant limiting segment.
Uniform continuity of $\Phi_\lambda$ on the compact path closure extends
the energy limit to all $\tau\to\infty$. Local uniform convergence of
$V_\tau$ and the vanishing normalization give
\eqref{eq:stochastic-energy-limits}.

Every translated limit now has constant energy and is stationary by
\eqref{eq:stochastic-profile-energy}. The closed-tail argument gives
\[
 \varnothing\neq\omega(Y)
 \subseteq\Ecal_\lambda\cap\{\Phi_\lambda=\ell^\varepsilon\},
 \qquad \omega(Y)\text{ compact and connected}.
\]
This is \eqref{eq:stochastic-omega}. If the distance to $\Ecal_\lambda$
failed to tend to zero, boundedness would give an accumulation point
outside this closed set, a contradiction.

\proofstep{Step 4}{Point convergence for empirical data}
For finite empirical active laws, \cref{lem:distance-geometry}\textup{(ii)} makes
$\Ecal_\lambda$ finite. Its nonempty connected subset $\omega(Y)$
is therefore a singleton. Boundedness then gives convergence of the
entire path. The limit along integer times is measurable; assigning a
fixed equilibrium on the exceptional null set defines the random vector
$y^{\varepsilon,*}$ in \eqref{eq:stochastic-endpoint}.
The inverse time change gives \cref{rem:stochastic-physical}.

\begin{proof}[Proof of \cref{cor:terminal-law}]
Let $Z=X_T^\varepsilon$ have law $\mu_T$ and set
$Y_\tau=X_{t_\tau}^\varepsilon$.
Use $M$ from \eqref{eq:stochastic-martingale} when $\varepsilon>0$,
and put $M=0$ when $\varepsilon=0$.
On each finite interval, the equation for $H=Y-M$ in Step~1
has a globally Lipschitz drift. Continuous dependence on the initial
state and the continuous noise path gives a jointly measurable
solution map.

For each fixed initial state, endpoint convergence follows from the
theorem just proved, or from \cref{thm:empirical-main} when
$\varepsilon=0$. Since $Z$ and $M$ are independent, Fubini's theorem
gives $Y_\tau\to Y_\infty\in\Ecal_\lambda$ almost surely for any
initial law $\mu_T$. The limit along integer times, assigned a fixed
equilibrium on the exceptional null set, is measurable. Set
$X_0^\varepsilon=Y_\infty$ and
$\mu_0^\varepsilon=\operatorname{Law}(X_0^\varepsilon)$.
Bounded convergence gives
\[
 \mathbb E f(X_t^\varepsilon)\longrightarrow
 \mathbb E f(X_0^\varepsilon),\qquad f\in C_b(\R^d),
\]
which proves narrow convergence of the laws. Since $\Ecal_\lambda$
is finite and $X_0^\varepsilon\in\Ecal_\lambda$ almost surely,
$\supp\mu_0^\varepsilon\subseteq\Ecal_\lambda$.
\end{proof}

\begin{proof}[Proof of \cref{rem:stochastic-perturbations}]
Fix a path on which $M$ converges and $r\in L^2(0,\infty)$, and put
$R_\tau=M_\tau+\int_0^\tau r_s\,ds$.
The random ODE for $\widetilde Y-R$ gives the strong solution as in Step~1.
Formula~\eqref{eq:duhamel} and the convolution estimate in
\cref{subsec:robustness-proof}, with kernel $e^{-a(\tau-s)}$,
give confinement. Moreover,
\[
 \sup_{0\leq h\leq S}|R_{\tau+h}-R_\tau|
 \leq\sup_{0\leq h\leq S}|M_{\tau+h}-M_\tau|
          +\sqrt S\,\|r\|_{L^2(\tau,\infty)}\longrightarrow0.
\]
The convergent function $M$ is uniformly continuous, while
$|\int_\tau^{\tau+h}r_s\,ds|\leq\sqrt h\,\|r\|_{L^2(0,\infty)}$.
Thus $R$ is uniformly continuous; boundedness of the drift along the
confined path gives the same property for $\widetilde Y$.
Lemma~\ref{lem:translated-profiles} gives limiting paths satisfying
\eqref{eq:stochastic-profile-energy}. Put
$\psi(\tau)=\Phi_\lambda(\widetilde Y_\tau)$.
If its lower and upper limits differed, the finite set
$\Phi_\lambda(\Ecal_\lambda)$ would allow a choice
\[
 \eta\in(\liminf\psi,\limsup\psi)
                 \setminus\Phi_\lambda(\Ecal_\lambda),
 \qquad b\in(\liminf\psi,\eta).
\]
Choose $s_j\to\infty$ with $\psi(s_j)=b$ and let
$\tau_j=\inf\{\tau>s_j:\psi(\tau)=\eta\}$. These times are finite.
Uniform continuity gives $h>0$ with $\tau_j-s_j>h$ for all $j$;
hence $\psi\leq\eta$ on $[\tau_j-h,\tau_j]$ and
$\psi(\tau_j)=\eta$. For a limiting path $z$ on $[0,h]$,
\[
 \Phi_\lambda(z_s)\leq\eta=\Phi_\lambda(z_h),
 \qquad 0\leq s\leq h.
\]
Dissipation forces constant energy and hence a constant path in
$\Ecal_\lambda$, contradicting the choice of $\eta$. Thus $\psi$ converges.
Dissipation makes all limiting paths stationary; finiteness and
connectedness give a single random endpoint. Substitution of
$r_\tau=(1+\varepsilon)t_\tau
e(\widetilde X_{t_\tau}^\varepsilon,t_\tau)$
and $d\tau=-dt/t$ gives \eqref{eq:stochastic-physical-error}.
\end{proof}
\section{Proofs of the technical lemmas}\label{sec:technical-proofs}

We retain the notation of Section~\ref{sec:proofs} and prove the five
lemmas in the order in which they were stated.

\subsection{Distance geometry and empirical rigidity}
\label{subsec:distance-technical}

\begin{proof}[Proof of \cref{lem:distance-geometry}]
\leavevmode\par\nobreak
\proofstep{Step 1}{Distance calculus and the separated envelope}
Apply the compact maximum rule
\cite[Theorem~10.31]{rockafellar1998variational} to
\[
 d_A(x)^2=|x|^2-2\max_{a\in A}
                  \left(\langle x,a\rangle-\frac{|a|^2}{2}\right).
\]
The distance formula and semiconcavity follow; in particular,
\begin{equation}\label{eq:distance-supergradient-inequality}
 d_A(y)^2\leq d_A(x)^2+\langle p,y-x\rangle+|y-x|^2,
 \qquad p\in\partial^C d_A^2(x);
\end{equation}
see \cite[Chapters~2--3]{cannarsa2004semiconcave}.
For $0\leq\lambda\leq1$, the convex sum rule applied to
$|\cdot|^2-d_k^2$ yields
$\widehat\partial\Phi_\lambda=\partial^C\Phi_\lambda$ and
$2$-semiconcavity of $\Phi_\lambda$. The Clarke sum rule gives
$\partial^C\Phi_\lambda\subseteq\widehat\partial\Phi_\lambda$
for signed coefficients \cite{rockafellar1970convex,clarke1989optimization}.
Compactness of the nearest-point sets gives a closed graph and local
boundedness of the envelope, hence upper semicontinuity.

\proofstep{Step 2}{Finiteness of empirical equilibria}
Suppose first that $\lambda\notin\{0,1\}$. Fix nonempty index sets
$I_k\subset\{1,\ldots,n_k\}$ and choose $i_k\in I_k$ for $k=1,2$.
Define
\[
 S=\operatorname{span}\{a_{k,i}-a_{k,i_k}:k=1,2,\ i\in I_k\}.
\]
A point realizing these nearest-index sets is equidistant from the
corresponding atoms. Expanding the squared distances gives
\begin{equation}\label{eq:pattern-equidistance}
 2\langle x,a_{k,i}-a_{k,i_k}\rangle
       =|a_{k,i}|^2-|a_{k,i_k}|^2,
 \qquad k=1,2,\quad i\in I_k.
\end{equation}
The solution set $M$, when nonempty, is an affine space with direction
$S^\perp$. The equilibrium condition also places $x$ in
\[
 C=\lambda\conv\{a_{1,i}:i\in I_1\}
       +(1-\lambda)\conv\{a_{2,i}:i\in I_2\}.
\]
Each component convex hull is contained in its chosen atom plus the
span of the corresponding differences. Hence
\[
 C\subseteq\lambda a_{1,i_1}+(1-\lambda)a_{2,i_2}+S.
\]
For two points $x,y\in M\cap C$, the affine constraints and this last
inclusion give, respectively,
\[
 x-y\in S^\perp,\qquad x-y\in S,
 \qquad\text{and therefore }x=y.
\]
Each nearest-index pattern contains at most one equilibrium, and there
are finitely many patterns. At $\lambda=0$ or $1$, the same argument
uses only the active support.
\end{proof}

\subsection{Heat-flow Hessians and the Li--Yau bound}
\label{subsec:hessian-technical}

\begin{proof}[Proof of \cref{lem:heat-hessian}]
The single-law Hessian identity is established in
\cite[Lemma~5.1 and its proof]{liuzuazua2025}; its signed combination
gives \eqref{eq:posterior-covariance}. Positivity of the covariance
matrices yields both curvature bounds. Since
$\partial_t\log u_k=\Delta\log u_k+|\nabla\log u_k|^2$,
taking the trace gives the Euclidean Li--Yau estimate
\cite{liyau1986parabolic}:
\begin{equation}\label{eq:heat-li-yau}
 |\nabla\log u_k|^2-\partial_t\log u_k
 =-\Delta\log u_k\leq\frac{d}{2t}.
\end{equation}
Compact supports bound the covariances uniformly. The Hessian
formulas therefore give the positive-time Lipschitz bounds, while
\eqref{eq:mean-shift-general} gives linear growth.
\end{proof}

\subsection{Gaussian localization and empirical rates}\label{subsec:localization-technical}

\begin{proof}[Proof of \cref{lem:compact-soft-min}]
The scalar limit is the fixed-measure Laplace--Varadhan principle
\cite[Section~4.3]{DemboZeitouni1998}. We give the locally uniform estimate and the posterior localization
needed here.

\proofstep{Step 1}{The locally uniform Laplace estimate}
For $h>0$, define
\[
 b_A(h)=\inf_{a\in A}u_0(B(a,h)).
\]
Cover $A$ by finitely many balls $B(a_j,h/2)$ centered on the support.
For each $a\in A$, some $j$ satisfies $|a-a_j|<h/2$, whence
$B(a_j,h/2)\subset B(a,h)$. Each covering ball has positive mass, so
\[
 b_A(h)\geq\min_j u_0(B(a_j,h/2))>0.
\]
Fix $x$, set $D=d_A(x)$, and choose $a_x\in\Pi_A(x)$.
Restricting the Gaussian integral to $B(a_x,h)$ and using
$|x-a|\geq D$ on $A$ gives
\[
 b_A(h)(4\pi t)^{-d/2}e^{-(D+h)^2/(4t)}
 \leq u(x,t)\leq(4\pi t)^{-d/2}e^{-D^2/(4t)}.
\]
Thus, for $\overline f=f-2dt\log(4\pi t)$,
\begin{equation}\label{eq:master-value-two-sided}
 0\leq\overline f(x,t)-d_A(x)^2
 \leq2d_A(x)h+h^2-4t\log b_A(h).
\end{equation}
On a compact set $B$, put $D_B=\sup_{x\in B}d_A(x)<\infty$. For each
fixed $h>0$, the preceding bound gives
\[
 \limsup_{t\to0^+}
 \sup_{x\in B}|\overline f(x,t)-d_A(x)^2|
 \leq2D_Bh+h^2.
\]
Letting $h\to0^+$ proves local uniform convergence of $\overline f$.
Since $2dt\log(4\pi t)\to0$, the same conclusion holds for $f$.

\proofstep{Step 2}{Posterior localization with a moving center}
Fix $x_*$, put $P=\Pi_A(x_*)$, and let
$U_\rho=\{a\in A:\dist(a,P)<\rho\}$ for $\rho>0$.
When $A\setminus U_\rho$ is nonempty, it is compact and disjoint from
$P$, so
\[
 \min_{a\in A\setminus U_\rho}
       \bigl(|x_*-a|^2-d_A(x_*)^2\bigr)>0.
\]
Continuity then gives a bounded neighborhood $\mathcal U$ of $x_*$ and
$\gamma>0$ such that
\[
 |x-a|^2\geq d_A(x)^2+\gamma,
 \qquad x\in\mathcal U,\quad a\in A\setminus U_\rho.
\]
Choose $h>0$ with $2Mh+h^2\leq\gamma/2$, where
$M=\sup_{x\in\mathcal U}d_A(x)$. The excluded Gaussian mass is at most
$(4\pi t)^{-d/2}e^{-(d_A(x)^2+\gamma)/(4t)}$. Dividing it by the lower
bound for $u(x,t)$ from Step~1 yields
\begin{equation}\label{eq:posterior-tail}
 \nu_t^x(A\setminus U_\rho)
 \leq b_A(h)^{-1}e^{-\gamma/(8t)},\qquad x\in\mathcal U.
\end{equation}
Let $C=\conv P$. The distance to $C$ is convex, is at most $\rho$ on
$U_\rho$, and is at most $\diam(A)$ on $A$. Jensen's inequality gives
\begin{equation}\label{eq:barycenter-tail-bound}
 \dist(m(x,t),C)
 \leq\int_A\dist(a,C)\,d\nu_t^x(a)
 \leq\rho+\diam(A)\nu_t^x(A\setminus U_\rho).
\end{equation}
When $U_\rho=A$, the same bound holds with zero excluded mass.
For $x_n\to x_*$ and $t_n\to0^+$, these estimates apply for all
sufficiently large $n$ and give
\[
 \limsup_{n\to\infty}\dist(m(x_n,t_n),C)\leq\rho.
\]
Letting $\rho\to0^+$ and using that $C$ is closed proves the assertion
about every cluster point of the bounded posterior means.

\proofstep{Step 3}{The empirical strict-cell estimate}
The empirical posterior probabilities are
\begin{equation}\label{eq:posterior-weights}
 p_{k,i}(x,t)=\frac{w_{k,i}e^{-|x-a_{k,i}|^2/(4t)}}
                   {\sum_jw_{k,j}e^{-|x-a_{k,j}|^2/(4t)}},
 \qquad m_k(x,t)=\sum_i p_{k,i}(x,t)a_{k,i}.
\end{equation}
There are finitely many competing atoms, and each has a positive
squared-distance gap at $x_0$. Continuity gives a neighborhood $U$ and
a common $\Delta>0$ such that
\[
 |x-a_{k,i}|^2-|x-a_{k,*}|^2\geq\Delta,
 \qquad x\in U,\quad i\neq *.
\]
Keeping the nearest Gaussian term in the denominator yields
\[
 p_{k,i}(x,t)
 \leq\frac{w_{k,i}}{w_{k,*}}
       \exp\!\left(-\frac{|x-a_{k,i}|^2-|x-a_{k,*}|^2}{4t}\right)
 \leq\frac{w_{k,i}}{w_{k,*}}e^{-\Delta/(4t)}.
\]
Since the posterior weights sum to one,
\[
 \begin{aligned}
 m_k(x,t)-a_{k,*}
   &=\sum_{i\neq *}p_{k,i}(x,t)(a_{k,i}-a_{k,*}),\\
 |m_k(x,t)-a_{k,*}|
   &\leq\left(\sum_{i\neq *}\frac{w_{k,i}}{w_{k,*}}
                         |a_{k,i}-a_{k,*}|\right)e^{-\Delta/(4t)}.
 \end{aligned}
\]
Combining the component bounds in the mean-shift identity proves
\eqref{eq:local-exponential-score} with $\eta=\Delta/4$.
Singleton supports contribute zero error; when every active support is
a singleton, the identity is exactly affine.
\end{proof}

\subsection{Passage to limiting paths}\label{subsec:profiles-technical}

\begin{proof}[Proof of \cref{lem:translated-profiles}]
\leavevmode\par\nobreak
\proofstep{Step 1}{Compactness of the translated curves}
Condition \eqref{eq:profile-forcing-increments} with $S=1$, together
with continuity on compact intervals, makes $R$ uniformly continuous.
Boundedness of $y$ and the posterior-mean formula give $Q_0<\infty$ with
\[
 |\nabla F_\lambda(y_s,t_s)|\leq Q_0\qquad(s\geq0).
\]
For $0\leq\tau\leq\tau'$, the integral equation gives
\[
 |y_{\tau'}-y_\tau|
 \leq\frac\kappa4Q_0(\tau'-\tau)+|R_{\tau'}-R_\tau|.
\]
The translates are bounded and equicontinuous. Arzel\`a--Ascoli and
diagonal extraction give, along a subsequence of any $\tau_n\to\infty$,
\[
 y_{\tau_n+\cdot}\longrightarrow z
 \quad\text{locally uniformly on }[0,\infty).
\]
Set $q_n(s)=\nabla F_\lambda(y_{\tau_n+s},t_{\tau_n+s})$. After a further
diagonal extraction, $q_n\rightharpoonup^*q$ in $L^\infty$ on every
bounded window. For $0\leq\tau<\tau'$, the translated equation reads
\[
 y_{\tau_n+\tau'}-y_{\tau_n+\tau}
 =-\frac\kappa4\int_\tau^{\tau'}q_n(s)\,ds
                  +R_{\tau_n+\tau'}-R_{\tau_n+\tau}.
\]
The final difference tends to zero by
\eqref{eq:profile-forcing-increments}. Passing to the limit gives
\[
 z_{\tau'}-z_\tau=-\frac\kappa4\int_\tau^{\tau'}q(s)\,ds,
 \qquad\dot z_s=-\frac\kappa4q(s)\quad\text{for a.e. }s.
\]
In particular, $z$ is locally Lipschitz.

\proofstep{Step 2}{The weak-star constraint}
For each fixed $s$, \cref{lem:compact-soft-min} and the posterior-mean
formula place every cluster point of $q_n(s)$ in
$C(s):=\widehat\partial\Phi_\lambda(z_s)$. Since $q_n(s)$ is bounded,
\[
 \dist(q_n(s),C(s))\longrightarrow0.
\]
The sets $C(s)$ are nonempty, compact, and convex, with a common bound
on each finite interval $I$. For $\xi\in\R^d$, their support function
\[
 h(s,\xi)=\max_{p\in C(s)}\langle\xi,p\rangle
\]
is measurable by upper semicontinuity of the envelope and continuity
of $z$. For every nonnegative $\varphi\in L^1(I)$,
\[
 \int_I\varphi(s)\langle\xi,q_n(s)\rangle\,ds
 \leq\int_I\varphi(s)h(s,\xi)\,ds
       +|\xi|\int_I\varphi(s)\dist(q_n(s),C(s))\,ds.
\]
The last integrand is dominated by a constant times $\varphi$.
Dominated convergence and weak-star convergence consequently imply
\[
 \int_I\varphi(s)
       \bigl(\langle\xi,q(s)\rangle-h(s,\xi)\bigr)\,ds\leq0.
\]
First take $\xi\in\mathbb Q^d$ and then use continuity in $\xi$.
On one full-measure set,
\[
 \langle\xi,q(s)\rangle\leq h(s,\xi)\quad(\xi\in\R^d),
 \qquad q(s)\in C(s).
\]
The last inclusion is the support-function characterization of closed
convex sets. Together with Step~1, it proves the lemma.
\end{proof}

\subsection{Deterministic and stochastic energy estimates}
\label{subsec:energy-technical}

We derive the posterior-information bound for part~\textup{(i)},
then prove part~\textup{(ii)} using transition-density free energy.

\begin{proof}[Proof of part~\textup{(i)} of \cref{lem:energy-estimates}]
\leavevmode\par\nobreak
\proofstep{Step 1}{The information identity}
For fixed $x$ and $t>0$, set
\[
 Z(x,t)=\int e^{-|x-a|^2/(4t)}\,du_0(a)>0,
 \qquad \overline f(x,t)=-4t\log Z(x,t).
\]
The posterior second moment $Q(x,t)$ is finite by Gaussian domination.
Taking the posterior expectation of
$\log(d\nu_t^x/du_0)=-|x-a|^2/(4t)-\log Z$ gives
\[
 \KL(\nu_t^x\|u_0)=-\frac{Q(x,t)}{4t}-\log Z(x,t).
\]
Differentiation under the Gaussian integral gives
$\partial_t Z/Z=Q/(4t^2)$. Consequently,
\begin{equation}\label{eq:KL-f-identity}
 \overline f(x,t)=Q(x,t)+4t\KL(\nu_t^x\|u_0),\qquad
 \partial_t\overline f(x,t)=4\KL(\nu_t^x\|u_0).
\end{equation}

\proofstep{Step 2}{The lower small-ball estimate}
Assume \eqref{eq:lower-mass}. Since $Q(x,t)\geq d_A(x)^2$,
\eqref{eq:master-value-two-sided} with $h=t$ gives, for
$0<t\leq t_B:=\min\{r_0,1\}$,
\begin{equation}\label{eq:compact-KL-estimate}
 \begin{aligned}
 0\leq4t\KL(\nu_t^x\|u_0)
 &=\overline f(x,t)-Q(x,t)\\
 &\leq\overline f(x,t)-d_A(x)^2
 \leq2t d_A(x)+t^2-4t\log(ct^\alpha).
 \end{aligned}
\end{equation}
On a compact set $B$, this gives $H_B(t)\leq C_B(1+|\log t|)$,
hence integrability near zero.
Continuity of the Gaussian integrals and their ratios makes $H_B$
bounded on compact positive-time intervals. For arbitrary $x$,
division by $4t$ also yields the global bound
\begin{equation}\label{eq:global-reference-information}
 \KL(\nu_t^x\|u_0)
 \leq\tfrac12d_A(x)+\tfrac t4-\log c-\alpha\log t
 \leq C(1+|x|+|\log t|).
\end{equation}

\proofstep{Step 3}{Potential variation}
For each compactly supported law, Jensen's inequality gives
\[
 d_k(x)^2\leq\overline f_k(x,t)
 \leq\int|x-b|^2\,du_{0,k}(b).
\]
Both bounds equal $|x|^2+O(|x|+1)$, uniformly in $t$. The identity $\lambda+(1-\lambda)=1$ and the mean-shift formula give
the first three bounds below. Moreover,
$0\leq4t\KL(\nu_{k,t}^x\|u_{0,k})=\overline f_k(x,t)-Q_k(x,t)
\leq\overline f_k(x,t)$, so \eqref{eq:KL-f-identity} also controls
the time derivative:
\begin{equation}\label{eq:free-energy-global-bounds}
 \begin{gathered}
 V_\tau(x)\geq\tfrac12|x|^2-C,\qquad
 |V_\tau(x)|\leq C(1+|x|^2),\\
 |\nabla V_\tau(x)|\leq C(1+|x|),\qquad
 |\partial_\tau V_\tau(x)|\leq C(1+|x|^2).
 \end{gathered}
\end{equation}
For $\lambda>1$, apply \eqref{eq:global-reference-information}
to the reference law and use nonnegativity of the target information.
For $0\leq\lambda\leq1$, both information terms are nonnegative.
The posterior identity therefore gives, for all $x$ and sufficiently
large $\tau$,
\begin{equation}\label{eq:free-energy-time-bound}
 \partial_\tau V_\tau(x)=-4t_\tau\Ical_\lambda(x,t_\tau)
 \leq Ct_\tau(1+|x|+|\log t_\tau|).
\end{equation}
The left-hand side is nonpositive in the convex regime.

For a nonempty compact set $B$, define
\[
 b_B(\tau)=
 \begin{cases}
  0,&0\leq\lambda\leq1,\\
  4(\lambda-1)t_\tau H_{2,B}(t_\tau),&\lambda>1.
 \end{cases}
\]
Nonnegativity of relative entropy and \eqref{eq:KL-f-identity}
give \eqref{eq:potential-time-control}. In the extrapolative case,
Step~2 and the substitution $t=Te^{-\tau}$ yield
\[
 \int_0^\infty b_B(\tau)\,d\tau
 =4(\lambda-1)\int_0^T H_{2,B}(t)\,dt<\infty.
\]
Thus $b_B$ is integrable, proving part~\textup{(i)}.
\end{proof}

\begin{proof}[Part~\textup{(ii)} of \cref{lem:energy-estimates}]
Fix $\varepsilon>0$ and a restart $(s,y)$, and put
\[
 \kappa=1+\varepsilon,\qquad a=\kappa/2,\qquad q=t_s=Te^{-s}.
\]
Write $Y=Y^{\varepsilon;s,y}$ and let $\rho_\tau$ be its density for
$\tau>s$. It satisfies
\begin{equation}\label{eq:transition-density-pde}
 \partial_\tau\rho_\tau
 =\frac\kappa4\operatorname{div}(\rho_\tau\nabla V_\tau)
       +\varepsilon t_\tau\Delta\rho_\tau,
 \qquad \rho_s=\delta_y.
\end{equation}
We first control the entropy produced by smoothing the point mass over a
short interval, and then use a free-energy inequality on the remaining
unit-time interval. All spatial integrals are over $\R^d$. Constants
below are independent of $y$ and sufficiently large $s$, with
$\varepsilon>0$ fixed.

\proofstep{Step 1}{Potential and restarted moment bounds}
Take $s$ large enough that $q\leq1/2$ and
\eqref{eq:free-energy-global-bounds} and
\eqref{eq:free-energy-time-bound} apply.
Along $Y$, put
\[
 z_\tau=\lambda m_1(Y_\tau,t_\tau)
          +(1-\lambda)m_2(Y_\tau,t_\tau)\in K_\lambda,
 \qquad R_\lambda=\max_{z\in K_\lambda}|z|.
\]
Since the drift equals $a(z_\tau-Y_\tau)$, the Duhamel formula
\eqref{eq:duhamel} yields, for $s\leq\tau\leq s+1$,
\[
 Y_\tau=e^{-a(\tau-s)}y
       +a\int_s^\tau e^{-a(\tau-s')}z_{s'}\,ds'
       +\int_s^\tau e^{-a(\tau-s')}
                     \sqrt{2\varepsilon t_{s'}}\,dW_{s'}.
\]
The drift convolution is bounded by
$R_\lambda(1-e^{-a(\tau-s)})$. As $t_{s'}\leq q$ on this interval,
It\^o's isometry gives
\begin{equation}\label{eq:restarted-moment-bounds}
 \begin{aligned}
 \mathbb E_{s,y}|Y_\tau|^2
  &=\int|x|^2\rho_\tau(x)\,dx\leq C(1+|y|^2),\\
 \mathbb E_{s,y}|Y_\tau-y|^2
  &\leq C\bigl[(\tau-s)^2(1+|y|^2)+q(\tau-s)\bigr].
 \end{aligned}
\end{equation}

\proofstep{Step 2}{Short-time smoothing and finite entropy}
Choose $h=q$ and compare $Y$, on $[s,s+h]$, with
\[
 dZ_\tau=-aZ_\tau\,d\tau+\sqrt{2\varepsilon t_\tau}\,dW_\tau,
 \qquad Z_s=y.
\]
By \eqref{eq:duhamel}, the law $\gamma_h$ of $Z_{s+h}$ is Gaussian
with mean $e^{-ah}y$ and covariance $\sigma_h^2I_d$, where
\[
 \sigma_h^2=2\varepsilon\int_s^{s+h}
       e^{-2a(s+h-s')}t_{s'}\,ds',
 \qquad c_\varepsilon q^2\leq\sigma_h^2\leq C_\varepsilon q^2.
\]
Thus the elapsed time and the diffusivity are both of order
$q$, giving a variance of order $q^2$.

Let $\mathbb P$ and $\mathbb Q$ denote the path laws of $Y$ and
$Z$ on $[s,s+h]$. Their drifts differ by $az_\tau(x)$, where
\[
 z_\tau(x)=\lambda m_1(x,t_\tau)
              +(1-\lambda)m_2(x,t_\tau)\in K_\lambda.
\]
Thus the normalized drift difference satisfies, uniformly in $x$,
\[
 |\eta_\tau(x)|^2
 :=\frac{a^2|z_\tau(x)|^2}{2\varepsilon t_\tau}
 \leq\frac{a^2R_\lambda^2}{2\varepsilon t_\tau}.
\]
The uniform bound ensures that the associated stochastic exponential
is a true martingale; by uniqueness, the changed law is that of $Y$.
Girsanov's relative-entropy identity
\cite{leonard2012girsanov},
followed by contraction to the terminal marginals, gives
\begin{equation}\label{eq:short-path-entropy}
 \begin{aligned}
 \KL(\rho_{s+h}\|\gamma_h)
 \leq\KL(\mathbb P\|\mathbb Q)
 &=\frac12\mathbb E_{s,y}
       \int_s^{s+h}|\eta_\tau(Y_\tau)|^2\,d\tau \\
& \leq\frac{a^2R_\lambda^2}{4\varepsilon}
       \frac{e^h-1}{q}
 \leq\frac{a^2R_\lambda^2 e^{1/2}}{4\varepsilon},\qquad h=q\leq1/2.
 \end{aligned}
\end{equation}

Expanding the Gaussian log-density, and dropping its nonpositive
quadratic term, we obtain
\begin{equation}\label{eq:short-density-entropy}
 \begin{aligned}
 \operatorname{Ent}(\rho_{s+h})
 &=\KL(\rho_{s+h}\|\gamma_h)
       +\int\rho_{s+h}\log\gamma_h\\
 &\leq \frac{a^2R_\lambda^2 e^{1/2}}{4\varepsilon}-\frac d2\log(2\pi\sigma_h^2)
 \leq C(1+|\log q|).
 \end{aligned}
\end{equation}
We also use the elementary lower bound, obtained by comparison with
the standard Gaussian,
\begin{equation}\label{eq:entropy-lower-moment}
 \operatorname{Ent}(\rho)
 \geq-\frac d2\log(2\pi)-\frac12\int|x|^2\rho(x)\,dx.
\end{equation}
The moment estimate and these two bounds make the entropy finite at
$s+h$. For any fixed $\tau_0>s$, the same Gaussian comparison over
$[s,\tau_0]$ gives finite entropy at $\tau_0$; choosing $h=q$ above
is what makes the bound uniform in $s,y$.

\proofstep{Step 3}{Free-energy dissipation}
On a fixed interval $[\tau_0,s+1]$, with $\tau_0>s$,
the transition density is smooth and positive, and the
diffusion is uniformly elliptic. Moreover, the drift has
at most linear growth and $\Delta V_\tau$ is bounded.
The Fokker--Planck equation
\eqref{eq:transition-density-pde} gives the local identity
\[
 \begin{aligned}
 \partial_\tau(\rho_\tau\log\rho_\tau)
 &=\operatorname{div}\!\left(
    \frac\kappa4\rho_\tau\log\rho_\tau\nabla V_\tau
    +\varepsilon t_\tau(1+\log\rho_\tau)\nabla\rho_\tau
    \right)\\
 &\quad+\frac\kappa4\rho_\tau\Delta V_\tau
       -\varepsilon t_\tau\frac{|\nabla\rho_\tau|^2}{\rho_\tau}.
 \end{aligned}
\]
Using the finite entropy and second-moment bounds from
Steps~1--2, the convex entropy truncation and spatial
localization argument of
\cite[Appendix~B]{liuzuazua2025}
justifies integration over $\R^d$ and yields
\begin{equation}\label{eq:finite-strip-entropy-balance}
 \begin{aligned}
 \operatorname{Ent}(\rho_\tau)
 +\int_{\tau_0}^{\tau}\varepsilon t_{s'}
   \underbrace{\int
       |\nabla\log\rho_{s'}|^2\rho_{s'}\,dx
   }_{\text{Fisher information}}\,ds'
 =\operatorname{Ent}(\rho_{\tau_0})
   +\frac\kappa4\int_{\tau_0}^{\tau}
       \int\Delta V_{s'}\,\rho_{s'}\,dx\,ds'.
 \end{aligned}
\end{equation}
The entropy lower bound
\eqref{eq:entropy-lower-moment} then implies finite
integrated Fisher information. Together with the linear
growth of $\nabla V_\tau$, this justifies the subsequent
integrations by parts.

Define the weighted free energy
\begin{equation}\label{eq:stochastic-free-energy}
 \theta_\tau=\frac{4\varepsilon t_\tau}{\kappa},\qquad
 \mathscr H_\tau=\int V_\tau\rho_\tau
                  +\theta_\tau\operatorname{Ent}(\rho_\tau).
\end{equation}
The Fokker--Planck equation \eqref{eq:transition-density-pde}
becomes
\[
 \partial_\tau\rho_\tau
 =\frac\kappa4\operatorname{div}
   \bigl(\rho_\tau\nabla(V_\tau+\theta_\tau\log\rho_\tau)\bigr).
\]
The time-derivative bound in \eqref{eq:free-energy-global-bounds}
and the second-moment estimate justify differentiation of the potential
term. Since $\theta_\tau'=-\theta_\tau$ and
$\int\partial_\tau\rho_\tau=0$, differentiation and integration
by parts yield
\begin{equation}\label{eq:conditional-free-energy-dissipation}
 \begin{aligned}
 \frac{d}{d\tau}\mathscr H_\tau
 &=
 -\frac\kappa4\int
       |\nabla V_\tau+\theta_\tau\nabla\log\rho_\tau|^2\rho_\tau\\
 &\quad+\int\partial_\tau V_\tau\,\rho_\tau
       -\theta_\tau\operatorname{Ent}(\rho_\tau).
 \end{aligned}
\end{equation}
The dissipation term is nonpositive even under signed mixing.
Moreover, $q/e\leq t_\tau\leq q$ on $[s,s+1]$, so
\eqref{eq:free-energy-time-bound},
\eqref{eq:restarted-moment-bounds}, and
\eqref{eq:entropy-lower-moment} give
\[
 \begin{aligned}
 \int\partial_\tau V_\tau\,\rho_\tau
 &\leq Cq(1+|\log q|)(1+|y|^2),\\
 -\theta_\tau\operatorname{Ent}(\rho_\tau)
 &\leq Cq(1+|y|^2).
 \end{aligned}
\]
Consequently,
\begin{equation}\label{eq:free-energy-increase}
 \frac{d}{d\tau}\mathscr H_\tau
 \leq Cq(1+|\log q|)(1+|y|^2)
 \qquad\text{for a.e. }\tau\in(s,s+1].
\end{equation}
Here $C$ is independent of the restart time $s$ and state $y$.

\proofstep{Step 4}{The conditional one-step bound}
By the global gradient bound in
\eqref{eq:free-energy-global-bounds},
\[
 |V_s(x)-V_s(y)|
 \leq C(1+|y|)|x-y|+C|x-y|^2.
\]
The second estimate in \eqref{eq:restarted-moment-bounds}, with $h=q$,
gives $\mathbb E_{s,y}|Y_{s+h}-y|^2\leq Cq^2(1+|y|^2)$.
This controls the spatial change in potential. Integrating
\eqref{eq:free-energy-time-bound} controls its time change and yields
\[
 \int V_{s+h}\rho_{s+h}
 \leq V_s(y)+Cq(1+|\log q|)(1+|y|^2).
\]
Since $\theta_{s+h}\leq Cq$, the upper entropy bound
\eqref{eq:short-density-entropy} gives
\begin{equation}\label{eq:free-energy-after-smoothing}
 \mathscr H_{s+h}
 \leq V_s(y)+Cq(1+|\log q|)(1+|y|^2).
\end{equation}
Integrating \eqref{eq:free-energy-increase} on $[s+h,s+1]$
provides the same upper bound for $\mathscr H_{s+1}$. Finally,
\eqref{eq:entropy-lower-moment} and
\eqref{eq:restarted-moment-bounds} imply
\[
 \begin{aligned}
 \mathbb E_{s,y}V_{s+1}(Y_{s+1})
 &=\mathscr H_{s+1}
           -\theta_{s+1}\operatorname{Ent}(\rho_{s+1})\\
 &\leq\mathscr H_{s+1}+Cq(1+|y|^2)\\
 &\leq V_s(y)+Cq(1+|\log q|)(1+|y|^2),
 \end{aligned}
\]
which is \eqref{eq:conditional-energy-step}, because $q=t_s$.
\end{proof}

\section{Conclusions and perspectives}\label{sec:conclusion}

We have identified the terminal asymptotics of deterministic and
stochastic dynamics driven by convex and extrapolative combinations
of logarithmic heat gradients. In logarithmic time, limits of translated
trajectories satisfy a generalized gradient inclusion associated with
the geometric potential $\Phi_\lambda$, a weighted sum of squared
support distances. For convex mixing, this is the Clarke subgradient
system of $\Phi_\lambda$. Under extrapolation, we retain the component
subdifferentials separately, since limiting directions need not belong
to $\partial^C\Phi_\lambda$. A componentwise chain rule
still gives dissipation. Together with posterior-entropy and conditional
free-energy estimates, it yields energy convergence and equilibrium
localization under the stated compact-support hypotheses.

For finite atomic laws, finiteness of the equilibrium set turns this
localization into point convergence for every deterministic
initialization and almost surely for each fixed stochastic
initialization, including projection-interface endpoints. Thus the
singular sampling dynamics has a well-defined terminal output in this
setting. At deterministic strict-cell endpoints, the expansion gives
an $O(\sqrt t)$ stopping error with an exponentially small remainder
after subtraction of the leading term; the exponent $1/2$ is sharp
when its coefficient is nonzero. The perturbation results show which
convergence properties survive square-integrable drift errors in
logarithmic time, while the nested-support corollary relates the
equilibrium geometry to
classifier-free guidance. Several questions remain.

\begin{enumerate}[label=\textup{(\roman*)}]
\item \emph{Stochastic rates and fluctuations.}
Energy control does not quantify the distance to a random endpoint.
Pathwise rates, fluctuation laws, and mean-square errors require
different estimates. Mean-square bounds also require control of paths
that enter their limiting regime late.
Figure~\ref{fig:empirical-2d-convergence} suggests square-root heat-time
scaling of the stochastic RMS distance to the marked local-minimum set
$\mathcal M_\lambda$, as measured in \eqref{eq:empirical-rms-diagnostic}.
A rate for the distance to the individual random endpoint would quantify
the stopping error of stochastic samplers.

\item \emph{Asymptotics at projection interfaces.}
At projection ties, several atoms can contribute at leading order,
and the positive gap underlying the strict-cell expansion vanishes.
Analyzing this competition and possible differences between tangential
and normal scales would extend the rate theory and clarify how support
geometry affects approach to the terminal sample.

\item \emph{Point convergence for general compact data.}
An equilibrium energy level can contain a continuum, so localization
alone need not yield a single limit. Additional nearest-point geometry
or gradient inequalities compatible with the componentwise inclusion
could establish terminal outputs beyond finite atomic laws.

\item \emph{Endpoint selection and terminal distributions.}
Equilibrium geometry does not determine reachability or selection
probabilities, which depend on the full dynamics, initialization, and
noise. For empirical data, determining these probabilities would identify
the weights of the terminal law and, for genuine conditional guidance,
clarify its relation to the intended conditional distribution.

\item \emph{Learning errors and time discretization.}
The weighted square-integrability conditions in
\cref{rem:physical-score-robustness,rem:stochastic-perturbations}
preserve deterministic localization and empirical endpoint convergence
for both dynamics, but need not preserve rates, selected equilibria,
or terminal laws. These conditions concern errors along generated
trajectories. Relating them to averaged training losses and numerical step sizes
requires estimates that remain valid through the singular regime.
\end{enumerate}

\section*{Statements and declarations}

\noindent\textbf{Data availability.}
All numerical examples are synthetic and specified in
Section~\ref{sec:numerics}. Numerical outputs and scripts are available
from the corresponding author upon reasonable request.

\medskip
\noindent\textbf{Competing interests.}
The authors declare that they have no conflict of interest.

\bibliographystyle{amsplain}
\bibliography{references}
\end{document}